\documentclass[a4paper,11pt,english]{smfart}
\usepackage{amsfonts}
\usepackage{amsmath,amsthm}
\usepackage{latexsym}
\usepackage{array}
\usepackage{amssymb}
\usepackage{graphicx}
\usepackage{enumerate}
\usepackage{verbatim}
\usepackage{subfigure}
\usepackage{float}
\usepackage[english]{babel}

\usepackage{color}
\definecolor{marin}{rgb}   {0.,   0.3,   0.7}
\definecolor{rouge}{rgb}   {0.8,   0.,   0.}
\definecolor{sepia}{rgb}   {0.8,   0.5,   0.}
\usepackage[colorlinks,citecolor=marin,linkcolor=rouge,
            bookmarksopen,
            bookmarksnumbered
           ]{hyperref}

\newcommand{ \black }{\color{black} }

\theoremstyle{plain}
\newtheorem{theorem}{Theorem}[section]
\newtheorem{lemma}[theorem]{Lemma}

\newtheorem{corollary}[theorem]{Corollary}
\newtheorem{definition}[theorem]{Definition} \theoremstyle{remark}

\newcommand {\aplt} {\ {\raise-.5ex\hbox{$\buildrel<\over\sim$}}\ }
\newcommand {\gplt} {\ {\raise-.5ex\hbox{$\buildrel>\over\sim$}}\ }

\newcommand{\dd}{\mathrm{d}}
\newcommand{\ee}{\mathrm{e}}

\newcommand{\R}{\mathbb{R}}

\newcommand{\Z}{\mathbb{Z}}

\newcommand{\V}{V_n(0)}
\newcommand{\Vs}{V_n(s)}
\newcommand{\Vx}{V_n(\xi)}

\newcommand{\boka}{\boldsymbol{\kappa}}
\newcommand{\bola}{\boldsymbol{\lambda}}
\newcommand{\bonu}{\boldsymbol{\nu}}
\newcommand{\bx}{\mathbf{x}}

\newcommand{\p}{}
\newcommand{\m}{-}

\makeatletter
\def\@makefnmark{\hbox{$\m@th^{\@thefnmark}$}}
\makeatother
\author{Marvin Kn\"oller}
\address{Fakult\"{a}t f\"{u}r Mathematik, Karlsruhe Institute of Technology,
Englerstr.~2, 76131 Karlsruhe, Germany}
\email{marvin.knoeller@icloud.com}

\author{Alexander Ostermann}
\address{Department of Mathematics, University of Innsbruck,
Technikerstr.~13, 6020 Innsbruck, Austria}
\email{alexander.ostermann@uibk.ac.at}

\author{Katharina Schratz}
\address{Fakult\"{a}t f\"{u}r Mathematik, Karlsruhe Institute of Technology,
Englerstr.~2, 76131 Karlsruhe, Germany}
\email{katharina.schratz@kit.edu}

\begin{abstract}
Standard numerical integrators suffer from an order reduction when applied to nonlinear Schr\"{o}dinger equations with low-regularity initial data. For example, standard Strang splitting requires the boundedness of the solution in $H^{r+4}$ in order to be second-order convergent in $H^r$, i.e., it requires the boundedness of four additional derivatives of the solution. We present a new type of integrator that is based on the variation-of-constants formula and makes use of certain resonance based approximations in Fourier space. The latter can be efficiently evaluated by fast Fourier methods. For second-order convergence, the new integrator requires two additional derivatives of the solution in one space dimension, and three derivatives in higher space dimensions. Numerical examples illustrating our convergence results are included. These examples demonstrate the clear advantage of the Fourier integrator over standard Strang splitting for initial data with low regularity.
\end{abstract}


\title[Fourier integrators for Schr\"{o}dinger equations]{A Fourier integrator for the cubic nonlinear Schr\"{o}dinger equation with rough initial data}

\begin{document}
\maketitle
\section{Introduction}\label{sec:intro}
The cubic nonlinear  Schr\"{o}dinger (NLS) equation
\begin{equation}\label{nls}
i \partial_t u(t,x) =\m \Delta  u(t,x) + \mu \vert u(t,x) \vert^{2}u(t,x),\qquad (t,x) \in \R \times \mathbb{T}^d,\quad  \mu \in \R,
\end{equation}
has been extensively studied in the numerical analysis literature. In particular, the error behavior of classical numerical schemes, e.g.~splitting methods and exponential integrators,  approximating solutions of~\eqref{nls} is nowadays  well understood: methods of arbitrary high order can be constructed for sufficiently smooth solutions.  More precisely, the global error of classical schemes for the cubic Schr\"{o}dinger equation~\eqref{nls} is dominated by terms of the form
\begin{equation}\label{gg}
\tau^\nu  (-\Delta)^\nu u(t)\quad \text{with} \quad  \nu \geq 0
\end{equation}
and $\tau$ denoting the time step size. This error behavior is caused by the fact that the nonlinear frequency interactions are neglected as the free Schr\"{o}dinger group $\mathrm{e}^{i t \Delta}$ in decoupled from the nonlinearity in classical schemes. Control of the error term~\eqref{gg} requires the boundedness of (at least) two additional derivatives of the solution for each order in $\tau$. \black For instance, the well known Strang splitting scheme applied to \eqref{nls} is second-order convergent in $H^r$ for solutions in $H^{r+4}$ for $r\geq 0$, see \cite{Lubich08}. Here $H^r$ denotes the classical Sobolev space on $\mathbb{T}^d$. For an extensive overview on splitting and exponential integration methods we further refer to \cite{HLW,HochOst10,HLRS10,McLacQ02}, and for their rigorous convergence analysis in the context of semilinear Schr\"{o}dinger equations we refer to \cite{BeDe02,CanG15,CCO08,CoGa12,Duj09,Faou12,JaLu00,Lubich08,Ta12} and the references therein.

From an analytical point of view a particular focus has been laid in the last decades on the investigation of the local and global well-posedness of the Schr\"{o}dinger equation \eqref{nls} in low regularity spaces, i.e., the existence of solutions for rough initial values, see~\cite{Bour93, Tao06}. The essential step in the local well-posedness analysis  lies in controlling  the underlying resonances of the equation. Inspired by these techniques we could recently develop a first-order exponential type integrator for the cubic Schr\"{o}dinger equation \eqref{nls} which allows convergence under far less regularity than required in~\eqref{gg} (set $\nu = 1$). The global error of the  new scheme is in particular driven by terms of the form
\begin{equation}\label{fo}
\tau (-\Delta)^{1/2} u(t)
\end{equation}
\black
which only requires the boundedness of one additional derivative of the solution, see \cite{OSch2017} for the detailed construction of the scheme and the error analysis.

In this manuscript we present for the first time \black a \emph{second-order Fourier integrator} for the Schr\"{o}dinger equation \eqref{nls} which allows us to lower the classical regularity assumptions \eqref{gg} and nevertheless obtain a convergence order $\nu > 1$. As in \cite{OSch2017}, the basic idea consists in filtering the linear flow by \black looking at the so-called \emph{twisted variable}
\begin{equation*}\label{twist}
v(t) = \ee^{\m i t\Delta } u(t)
\end{equation*}
which satisfies the \emph{twisted} NLS equation
\begin{equation*}\label{nichtegal}
i \partial_t v(t) = \mu \ee^{\m i t \Delta} \Big[ \bigl\vert  \ee^{\p i t \Delta} v(t)\bigr\vert^2  \ee^{\p i t \Delta} v(t)\Big]
\end{equation*}
with mild solution given by
\begin{equation*}\label{vsold}
v(t_n+\tau) = v(t_n) - i \mu \int_0^\tau \ee^{\m i (t_n+s) \Delta} \Big[ \bigl\vert  \ee^{\p i (t_n+s) \Delta} v(t_n+s)\bigr\vert^2  \ee^{\p i (t_n+s) \Delta} v(t_n+s)\Big]\dd s.
\end{equation*}
Inserting this expression iteratively into its right-hand side gives an expansion of the twisted variable in terms of the step size $\tau$. In one space dimension, the first term of this expansion is given by
$$
- i \mu \int_0^\tau \ee^{\m i (t_n+s) \partial_x^2} \Big[ \bigl\vert  \ee^{\p i (t_n+s) \partial_x^2} v(t_n)\bigr\vert^2  \ee^{\p i (t_n+s) \partial_x^2} v(t_n)\Big]\dd s.
$$
Using standard Fourier techniques, the term can be expressed as
$$
-i\mu \ee^{-i t_n \partial_x^2} \int_0^\tau \sum_{k_1,k_2,k_3 \in \Z}  \ee^{i s \big[ \left(k_1+k_2+k_3\right)^2 + k_1^2- \left(k_2^2+k_3^2\right) \big] } \hat{\overline{v}}_{k_1}\hat v_{k_2} \hat v_{k_3} \ee^{i (k_1+k_2+k_3) x} \,\dd s
$$
and integrated exactly without much effort. The same applies to the other terms of the expansion. Thus, this procedure eventually leads to a numerical scheme of arbitrary high order that would not require boundedness of additional derivatives of the solution. However, two problems are inherent to this procedure, namely resonances and computational costs. The first problem consists in finding all $k_1, k_2, k_3 \in\Z$ such that
$$
\left(k_1+k_2+k_3\right)^2 + k_1^2- \left(k_2^2+k_3^2\right) = 0.
$$
These resonant frequencies have to be treated separately, before carrying out the integration
$$
-i\mu \ee^{-i t_n \partial_x^2}\sum_{k_1,k_2,k_3 \in \Z} \frac{\ee^{i \tau  \big[\left(k_1+k_2+k_3\right)^2 + k_1^2- \left(k_2^2+k_3^2\right) \big]}-1}{ \left(k_1+k_2+k_3\right)^2 + k_1^2- \left(k_2^2+k_3^2\right) }\, \hat{\overline{v}}_{k_1}\hat v_{k_2} \hat v_{k_3} \ee^{i (k_1+k_2+k_3) x}.
$$
The second problem consists in evaluating these expressions. This pointwise evaluation is unduly expensive as it requires $\mathcal O(N^3)$ operations, in general.

In our previous work \cite{OSch2017} we took another approach leading to a scheme which only required $\mathcal O(N \log N)$ operations and still allowed reduced regularity assumptions for the solution. We intend to follow the same ideas here. The second-order low regularity construction, however, is much more involved. This is mainly due to the structure of resonances. \black In particular, a straightforward higher-order extension of the ideas for the first-order scheme developed in \cite{OSch2017} would only lead to an unstable scheme. Therefore, new techniques need to be applied.

Based on a rigorous higher-order resonance analysis we construct the second-order Fourier integrator
\begin{equation}\label{uuintro}
\begin{aligned}
u^{n+1}
& = \ee^{ i \tau \partial_x^2}  \Bigl(\ee^{i \mu \tau \vert u^n\vert^2}u^n - i \mu \bigl(J_{1,x}^\tau\left( u^n\right) + J_{2,x}^\tau \left( u^n\right)\bigr)\Bigr),
\end{aligned}
\end{equation}
where $u^n$ approximates the solution of \eqref{nls} at time $t = t_n$. For the precise structure of the correction terms see Lemmas~\ref{lem:int1} and \ref{lem:int2} below. The new scheme \eqref{uuintro} firstly allows second-order approximations of \eqref{nls} in $H^r$ for solutions in $H^{r+2}$ in the one dimensional setting $d=1$ with a global error  driven by terms of the form
\begin{equation*}
\tau^2  \partial_{xx} u(t),
\end{equation*}
\black
see Theorem~\ref{thm-1d} for the precise convergence estimate. This in particular includes the important class of classical solutions in $L^2$ (set $r =0$) and extends the first-order bound \eqref{fo} in a natural and expected way. However, due to the complicated structure of resonances we face an order reduction down to $3/2$ in the higher-dimensions setting $d > 1$, see Section~\ref{sect:conv-dd} and in particular Theorem~\ref{thmd-a} for the precise convergence analysis.

Note that our analysis is focused on periodic Schr\"odinger equations \eqref{nls} posed on the $d$-dimensional torus $\mathbb{T}^d$. It is an interesting future research goal to develop and analyse low regularity integrators also on general bounded domains $\Omega \subset \mathbb{R}^d$ using different spatial discretization techniques such as finite differences, finite elements, finite volume or discontinuous Galerkin methods. \black

The paper is organized as follows. The first part is devoted to the construction of the Fourier integrator in one space dimension. In Section~\ref{sect:const-1d}, we illustrate the main ideas and study in particular the local error of the method by analysing each of the approximations carried out. The convergence results in dimension $d=1$ are given in Section~\ref{sect:conv-1d}. This section also contains the necessary stability estimates. The extension of our Fourier integrator to higher dimensions is given in Section~\ref{sect:const-dd}, its convergence properties are summarized in Section~\ref{sect:conv-dd}. We conclude in Section~\ref{sect:num} with some numerical examples that illustrate our convergence results and the performance of the new integrator. We conclude in Section~\ref{sect:conclusions}.

We close this section with some remarks. Let $H^r= H^r(\mathbb{T}^d)$ denote the classic Sobolev space of $H^r$ functions on the $d$ dimensional torus $\mathbb T^d$. Its norm $\Vert \cdot \Vert_r$ is defined for $v(\mathbf x) = \sum_{\mathbf k\in\Z^d} \hat v_{\mathbf k} \ee^{i\mathbf k\cdot\mathbf x}$ by
$$
\|v\|_r^2 = \sum_{\mathbf k\in\Z^d} (1+|\mathbf k|)^{2r} |\hat v_{\mathbf k}|^2,
$$
where
$$
\hat v_{\mathbf k} = \frac{1}{(2\pi)^d} \int_{\mathbb{T}^d} \mathrm{e}^{-i \mathbf k \cdot \mathbf x} v(\mathbf x) \mathrm{d} \mathbf x
$$
for $\mathbf k \in \mathbb{Z}^d$ denote the Fourier coefficients associated with $v$. Here we have set
$$
\mathbf k \cdot \mathbf x = k_1 x_1 + \ldots + k_d x_d \quad \text{and} \quad \vert \mathbf k\vert^2 = k_1^2 + \ldots + k_d^2.
$$
\black
Throughout the paper, we will exploit the well-known bilinear estimate
\begin{equation}\label{bil}
\Vert f g \Vert_r \leq c_{r,d} \Vert f \Vert_r \Vert g \Vert_r,
\end{equation}
which holds for $r>d/2$ and some constant $c_{r,d}>0$, and we will make frequent use of the isometric property of the free Schr\"odinger group $\ee^{i t \Delta}$
\begin{equation}\label{iso}
\Vert \ee^{i t \Delta}f \Vert_r = \Vert f \Vert_r
\end{equation}
for all $f \in H^r$ and $t \in \R$. Unless otherwise stated, $c$ denotes a generic constant that is allowed to depend on $d$ and $r$.

\section{Construction of the scheme in one dimension $(d=1)$}\label{sect:const-1d}

To illustrate the ideas in the construction and analysis of the Fourier integrator we first focus on one spatial dimension. For $f \in L^2(\mathbb{T})$ we will denote its Fourier expansion by $f(x) = \sum_{k\in \Z} \hat{f}_k \ee^{i k x}$. Furthermore, we  define a regularization of $\partial_x^{-1}$ through its action in Fourier space by
\begin{subequations}\label{conti}
\begin{equation}\label{conti-a}
(\partial_x^{-1})_k  =
\left\{
\begin{array}{ll}
(ik)^{-1} &\mbox{if} \quad k \neq 0\\
0 & \mbox{if}\quad k = 0
\end{array}
\right.,  \text{ i.e.,} \quad\partial_x^{-1} f(x) =\sum_{k\in \Z\setminus \{ 0\} } (ik)^{-1} \hat{f}_k \ee^{ikx},
\end{equation}
and consequently
\begin{equation}\label{conti-b}
\left(\frac{\ee^{i t \partial_x^2} - 1}{i t \partial_x^2} \right)_{k\neq 0}  =
\frac{\ee^{-i t k^2} - 1}{- i t k^2}, \qquad \left(\frac{\ee^{i t \partial_x^2} - 1}{i t \partial_x^2} \right)_{k=0} = 1.
\end{equation}
\end{subequations}

In the remainder of this section we assume that $r>1/2$ such that the bilinear estimate \eqref{bil} holds. In the construction of our numerical scheme we will frequently employ the following estimate.
\begin{lemma}\label{lem:xp-est}
For all $\beta \in \R$ and $0\le \gamma \le 1$ it holds that $\left|\ee^{i \beta} - 1\right|\le 2^{1-\gamma}\left|\beta \right|^\gamma$.
\end{lemma}
\begin{proof}
The assertion follows at once from a convex combination of the following estimates
$$
\left|\ee^{i \beta} - 1\right|\le 2, \qquad\quad \left|\ee^{i \beta} - 1\right| \le \left|\int_0^\beta i\ee^{is} \dd s\right| \le \left|\beta \right|,
$$
which both hold for all $\beta\in\R$.
\end{proof}
In the construction of our numerical scheme, we will employ the following expansion.
\begin{lemma}\label{lem:v-order0}
For all $0\le \gamma \le 1$ and $v\in H^{r+2\gamma}$ it holds
\begin{equation}\label{expR}
\ee^{\pm i s \partial_x^2} v = v + R \quad\text{with} \quad \|R\| \leq 2^{1-\gamma} \vert s \vert^\gamma  \Vert v \Vert_{r+2\gamma}.
\end{equation}
\end{lemma}
\begin{proof}
The identity
\begin{equation*}
\ee^{i \beta} = 1 + \vert \beta \vert^\gamma  \frac{ \ee^{i \beta} - 1}{\vert \beta \vert^\gamma}
\end{equation*}
allows us to expand terms of type  $\ee^{i s \partial_x^2}v$ as
\begin{equation*}
\ee^{\pm i s \partial_x^2} v = v + \vert s\vert^\gamma \!\!\sum_{k \in \Z\setminus\{0\}}
\frac{ \ee^{\mp i s k^2} - 1}{ \vert s \vert^\gamma \vert  k \vert^{2\gamma}} \vert  k\vert^{2\gamma}\hat{v}_k \,\ee^{i k x}.
\end{equation*}
The sum is readily estimated with the help of Lemma~\ref{lem:xp-est}.
\end{proof}

Recall that the mild solution in one space dimension reads
\begin{equation}\label{vsol}
\begin{aligned}
v(t_n+\tau) = v(t_n) - i \mu  \ee^{\m i t_n \partial_x^2} \int_0^\tau \ee^{\m i s \partial_x^2} \Big[ \bigl\vert  \ee^{\p i (t_n+s) \partial_x^2} v(t_n+s)\bigr\vert^2  \ee^{\p i (t_n+s) \partial_x^2} v(t_n+s) \Big]\dd s.
\end{aligned}
\end{equation}
The purpose of this section is to derive a globally second-order approximation to $v(t_n+\tau)$. For this aim we need to approximate the integrand in the appearing integral up to a local error of order $s^2$, which requires a first-order approximation of $v(t_n+s)$. Such an approximation will be derived in Section~\ref{subsec:first-order}.

To simplify the presentation we will henceforth use the notation
\begin{equation}\label{vast}
\Vs = \ee^{i t_n \partial_x^2} v(t_n+s),
\end{equation}
and we will denote remainder terms as follows.
\begin{definition}
Let $R=R(v,t,s)$ be a term that depends on the function values $v(t+\xi)$ for $0\le\xi\le s$. We say that $R$ is in the class of remainders $\mathcal R_\beta(s^\alpha)$ if and only if the bound
\begin{equation}\label{eq:Rest}
\|R(v,t,s)\|_r \le C s^\alpha
\end{equation}
holds with a constant $C$ that only depends on $d, r, \mu$, and $\sup_{0\le\xi\le s}\|v(t+\xi)\|_{r+\beta}$.
\end{definition}
Instead of $f=g+R$ with $R\in\mathcal R_\beta(s^\alpha)$ we will also write $f=g+\mathcal R_\beta(s^\alpha)$ for short.

\subsection{A first-order approximation of $v(t_n+s)$}\label{subsec:first-order}

As already explained above, we require a first-order approximation to $v(t_n+s)$. Such an approximation was already derived in our previous work \cite{OSch2017} using initial data in $H^{r+1}$. For our second-order scheme, however, we will finally employ $H^{r+2}$ initial data. Therefore, there is more freedom in getting a simpler approximation. This will be carried out next.

Employing the notation \eqref{vast}, the mild solution \eqref{vsol} at time $t_n+s$ reads
\begin{equation}\label{vs}
\begin{aligned}
v(t_n+s) = v(t_n) - i \mu \ee^{\m i t_n \partial_x^2} \int_0^s \ee^{\m i \xi \partial_x^2} \Big[ \big\vert  \ee^{\p i \xi \partial_x^2} \Vx \big\vert^{2}  \ee^{\p i \xi \partial_x^2} \Vx \Big]\dd \xi.
\end{aligned}
\end{equation}
Using the identity
\begin{align*}
\vert v_1 \vert^2 v_1 =\left (\overline{v}_1-\overline{v}_2\right) v_1^2 + \overline{v}_2 \left(v_1-v_2\right) \left(v_1 + v_2\right) + \vert v_2 \vert^2 v_2
\end{align*}
with $v_1 = \ee^{\p i \xi \partial_x^2} \Vx$ and $v_2 = \ee^{\p i \xi \partial_x^2} \V$ in the integral in \eqref{vs} allows us to express $v(t_n+s)$ as follows:
\begin{align}\label{vex1}
v(t_n+s)  = v(t_n) -i\mu  \ee^{\m i t_n \partial_x^2} J^s\big(\V\big) + R_1(v,t_n,s)
\end{align}
with  the integral
\begin{equation}\label{I1}
J^s(w)  = \int_0^s \ee^{\m i \xi  \partial_x^2} \Big[ \big\vert  \ee^{\p i \xi \partial_x^2} w \big\vert^{2}  \ee^{\p i \xi \partial_x^2} w \Big]\dd \xi
\end{equation}
and the remainder $R_1(v,t_n,s)$  given by
\begin{equation}\label{r1}
\begin{aligned}
R_1(v,t_n,s) =  - i \mu   \ee^{\m i t_n \partial_x^2}  \int_0^s \ee^{\m i \xi \partial_x^2} \Big[\left(\ee^{\m i \xi \partial_x^2} \big(\overline \Vx -\overline{\V} \big) \right)  \left(\ee^{\p i\xi \partial_x^2} \Vx \right)^2\Big]\dd \xi\qquad\\
- i \mu  \ee^{\m i t_n \partial_x^2}   \int_0^s \ee^{\m i \xi  \partial_x^2} \Big[\left(\ee^{\m i \xi  \partial_x^2} \overline \V \right)  \left(\ee^{\p i \xi  \partial_x^2} \big(\Vx -\V \big)\right) \left(\ee^{\p i \xi \partial_x^2} \big(\Vx +\V \big)\right)\Big]\dd \xi.
\end{aligned}
\end{equation}
The following lemma shows that the remainder is of second order in $s$.

\begin{lemma}\label{lem:bound_R2}
The remainder term $R_1(v,t_n,s)$ in \eqref{vex1} belongs to class $\mathcal R_0(s^2)$. In particular, it satisfies the bound
\begin{equation*}
\|R_1(v,t_n,s)\|_r \le  c s^2 \mu^2 \sup_{0\le\xi\le s}\|v(t_n+\xi)\|_r^5
\end{equation*}
with a constant $c$ that only depends on $c_{r,d}$.
\end{lemma}

\begin{proof}
Starting from the representation in \eqref{vsol} we readily obtain by employing the bilinear estimate \eqref{bil} together with the isometric property \eqref{iso} that
\begin{equation}\label{vb1}
\Vert \Vx - \V \Vert_r = \Vert v(t_n+\xi) - v(t_n) \Vert_r  \leq  c \xi \vert \mu \vert \sup_{0 \leq \eta \leq \xi} \Vert v(t_n+\eta)\Vert_r^3.
\end{equation}
Plugging \eqref{vb1} into \eqref{r1} we thus obtain  (using again \eqref{bil} and \eqref{iso}) the second-order error bound
\begin{equation}\label{r1bound}
\begin{aligned}
\Vert R_1(v,t_n,s)  \Vert_r & \leq c\vert \mu\vert \sup_{0 \leq \xi \leq s} \Vert v(t_n+\xi)\Vert_r^2 \int_0^s  \Vert v(t_n+\eta)-v(t_n)\Vert_r  \dd \eta \\
 & \leq c s^2 \mu^2 \sup_{0 \leq \xi \leq s} \Vert v(t_n+\xi)\Vert_r^5
\end{aligned}
\end{equation}
as desired.
\end{proof}

We still have to find an appropriate approximation to the integral $J^s\bigl(\V\bigr)$. The expansion~\eqref{expR} (together with \eqref{bil} and \eqref{iso}) implies that
\begin{equation}\label{expIs}
\begin{aligned}
J^s\bigl(\V\bigr) & = \int_0^s   \Big(\left \vert \V\right\vert^2 \V + \mathcal R_{2\gamma}(s^\gamma)\Big) \dd \xi\\
& = s \left \vert \V\right\vert^2 \V + \mathcal R_{2\gamma}(s^{1+\gamma}).
\end{aligned}
\end{equation}

Plugging this expansion into \eqref{vex1} proves the following result.

\begin{lemma}\label{lem:order-1}
For $0 \leq \gamma \leq 1$ and $v\in H^{r+2\gamma}$ it holds that
\begin{equation*}
v(t_n+s) = v(t_n) - i \mu s \,\ee^{-i t_n \partial_x^2}\left( \left \vert \V \right\vert^2 \V\right) + \mathcal R_{2\gamma}(s^{1+\gamma}).
\end{equation*}
\end{lemma}

As a consequence, we obtain for $0 \leq \gamma \leq 1$ that
\begin{equation}\label{vex2}
\Vs = \V \black - i \mu s \left \vert \V \right\vert^2 \V +  \mathcal R_{2\gamma}(s^{1+\gamma}),
\end{equation}
where we used once more the isometric property \eqref{iso}.

\subsection{A second-order approximation of $v(t_n+\tau)$}\label{sec:sec-order}

Next we insert the expansion \eqref{vex2} into the mild solution $v(t_n+\tau)$  given in \eqref{vs}. \black For $0\le \gamma\le 1$ and $v\in H^{r+2\gamma}$ this yields that
\begin{equation*}
\begin{aligned}
v(t_n+\tau) &= v(t_n) - i \mu  \ee^{\m i t_n \partial_x^2}  \int_0^\tau \ee^{\m i s\partial_x^2} \Big[\left(\ee^{\m i s\partial_x^2} \overline \V \right)  \left( \ee^{\p i s \partial_x^2} \V \right)^2\Big]\dd s \\
&\qquad + \mu^2 \ee^{\m i t_n \partial_x^2} \int_0^\tau s \cdot\ee^{\m i s \partial_x^2} \Big[\left(\ee^{\m i s \partial_x^2} \big(  \vert \V \vert^2 \overline{\V} \big)\right)  \left( \ee^{\p i s \partial_x^2} \V  \right)^2\Big]\dd s \\
&\qquad - 2\mu^2 \ee^{\m i t_n \partial_x^2}  \int_0^\tau s \cdot \ee^{\m i s \partial_x^2} \Big[\left( \ee^{\p i s \partial_x^2} \big(\vert \V \vert^2 \V \big)\right)\left|\ee^{i s \partial_x^2}  \V \right|^2 \Big]\dd s  + \mathcal R_{2\gamma}(\tau^{2+\gamma}).
\end{aligned}
\end{equation*}

Employing the expansion given in \eqref{expR} in the second and third integral (which are of order $\tau^2$) furthermore implies that
\begin{equation}
\begin{aligned}\label{vpre}
v(t_n+\tau) &= v(t_n) - i \mu  \ee^{\m i t_n \partial_x^2}  \int_0^\tau \ee^{\m i s\partial_x^2} \Big[\left(\ee^{\m i s\partial_x^2} \overline \V \right)  \left( \ee^{\p i s \partial_x^2} \V \right)^2\Big]\dd s \\
&\qquad - \mu^2  \ee^{- i t_n \partial_x^2}\int_0^\tau s \,\left \vert \V \right \vert^4 \V\,\dd s + \mathcal R_{2\gamma}(\tau^{2+\gamma})\\
&= v(t_n) - i \mu \ee^{- i t_n \partial_x^2}  J^\tau_x(\V)
- \mu^2 \frac{\tau^2}{2}\ee^{- i t_n \partial_x^2} \left(\left \vert \V \right \vert^4 \V
\right) + \mathcal R_{2\gamma}(\tau^{2+\gamma})
\end{aligned}
\end{equation}
with
\begin{equation*}
J^\tau_x\bigl(\V\bigr) = \int_0^\tau \ee^{\m i s\partial_x^2} \Big[\left(\ee^{\m i s\partial_x^2} \overline \V \right)  \left( \ee^{\p i s \partial_x^2} \V \right)^2\Big]\dd s.
\end{equation*}

We still have to find a refined approximation to the integral $J^\tau_x\bigl(\V\bigr)$. \black In a first step we express $J^\tau_x (v)$ for $v\in H^{r+2\gamma}$ by its Fourier series:\black
\begin{equation}\label{J}
J^\tau_x(v) = \int_0^\tau \sum_{k_1,k_2,k_3 \in \Z}  \ee^{i s \big[ \left(k_1+k_2+k_3\right)^2 + k_1^2- \left(k_2^2+k_3^2\right) \big] } \hat{\overline{v}}_{k_1}\hat v_{k_2} \hat v_{k_3} \ee^{i (k_1+k_2+k_3) x} \dd s.
\end{equation}
In order to obtain a \emph{numerically stable} second-order approximation of the integral $J^\tau_x(v)$, which in particular \emph{does not contain explicit derivatives}, we need to find an appropriate approximation to the exponential
\begin{equation}\label{expi}
\ee^{i s \big[ \left(k_1+k_2+k_3\right)^2 + k_1^2- \left(k_2^2+k_3^2\right) \big] }.
\end{equation}
For this aim, we will employ the following decomposition.
\begin{lemma}\label{mani}
For all $\alpha, \beta \in \R$ we have that
\begin{align*}\label{mani}
\ee^{i (\alpha + \beta)} = \ee^{i \alpha } + \ee^{i \beta} -1 + R(\alpha,\beta) \qquad \text{with} \quad \left \vert R(\alpha,\beta)\right\vert \leq 2^{2-\delta-\theta} \vert \alpha\vert^\delta \vert \beta \vert^\theta \quad \text{for} \quad 0\leq \delta, \theta \leq 1.
\end{align*}
\end{lemma}

\begin{proof}
The assertion follows from
\begin{equation*}
\begin{aligned}
\ee^{i (\alpha + \beta)}
& =  \vert \alpha\vert^\delta \vert \beta\vert^\theta \frac{ \big(\ee^{i \alpha } -1\big)}{\vert \alpha\vert^\delta} \frac{\big(  \ee^{i \beta } -1\big)}{\vert \beta \vert^\theta}
+ \ee^{i \alpha} + \ee^{i \beta} - 1
\end{aligned}
\end{equation*}
together with Lemma~\ref{lem:xp-est}.
\end{proof}
The \emph{key relation}
\begin{equation}\label{key}
\left(k_1+k_2+k_3\right)^2 + k_1^2- \left(k_2^2+k_3^2\right) = 2 k_1 \left(k_1+k_2+k_3\right) + 2k_2k_3
\end{equation}
now allows us to write the exponential \eqref{expi} as follows
\begin{equation}\label{expE}
\begin{aligned}
\ee^{i s \big[ \left(k_1+k_2+k_3\right)^2 + k_1^2- \left(k_2^2+k_3^2\right) \big] }
 & = \ee^{ 2 i s k_1 \left(k_1+k_2+k_3\right)} + \ee^{ 2 i s k_2 k_3} - 1 +  R_2^s( k_1,k_2,k_3)
\end{aligned}
\end{equation}
with the remainder
\[
R_2^s( k_1,k_2,k_3) = \left(  \ee^{ 2 i s k_1(k_1+k_2+k_3)} -1 \right) \left(  \ee^{ 2 i s k_2 k_3}- 1 \right).
\]
Thanks to Lemma \ref{mani} (with $\delta = \theta = \gamma$) the remainder satisfies the following bound
\begin{equation}\label{BRk3}
\begin{aligned}
\left \vert R_2^s( k_1,k_2,k_3) \right \vert
& \leq 4^{1-\gamma} s^{2 \gamma} \left\vert  2k_1 (k_1+k_2+k_3) \right \vert^\gamma \left\vert 2 k_2 k_3 \right \vert^\gamma \\
& \leq  2\, s^{2\gamma} \sum_{\substack{\ell,j,m \in \{1,2,3\}\\\ell\neq j\neq m}} \vert k_\ell^2  k_j  k_m\vert^\gamma
\end{aligned}
\end{equation}
for $0\le\gamma\le 1$. Plugging the expansion \eqref{expE} into \eqref{J} yields that
\begin{equation}\label{J1}
J^\tau_x(v) = \sum_{k\in\Z}\sum_{\substack{k_1,k_2,k_3 \in \Z\\k_1+k_2+k_3 = k}} \hat{\overline{v}}_{k_1}\hat v_{k_2} \hat v_{k_3}
\ee^{i k x}  \int_0^\tau \left(\ee^{ 2 i s k_1 k} + \ee^{ 2 i s k_2 k_3} - 1\right)\dd s + R_3^\tau(v),
\end{equation}
with the remainder
\begin{equation*}
R_3^\tau(v) = \sum_{k_1,k_2,k_3 \in \Z}  \int_0^\tau R_2^s( k_1,k_2,k_3) \,\hat{\overline{v}}_{k_1}\hat v_{k_2} \hat v_{k_3} \ee^{i (k_1+k_2+k_3) x}\,\dd s.
\end{equation*}

Next, we give a bound on this term.

\begin{lemma}\label{lem:int-rem}
The remainder term $R_3^\tau(v)$ in \eqref{J1} belongs to class $\mathcal R_{2\gamma}(\tau^{1+2\gamma})$.
\end{lemma}

\begin{proof}
Thanks to the bound on $R_2^s( k_1,k_2,k_3)$ given in \eqref{BRk3} we obtain that
\begin{equation}\label{r2ga}
\begin{aligned}
\left \Vert R_3(v) \right \Vert_r^2 & \leq  c \sum_{k \in \Z}
\left(1+\vert k \vert\right)^{2r}\Biggl( \sum_{\substack{k_1+k_2+k_3 = k\\k_1,k_2,k_3 \in \Z}}  \int_0^\tau  s^{2\gamma} \dd s\sum_{\substack{\ell,j,m \in \{1,2,3\}\\\ell\neq j\neq m}} \vert k_\ell^2  k_j  k_m\vert^\gamma\vert \hat{v}_{k_\ell}\vert \vert \hat{v}_{k_j}\vert \vert \hat{v}_{k_m}\vert \Biggr)^2
\end{aligned}
\end{equation}
for some constant $c>0$ and $0 \leq \gamma \leq 1$ which implies that the remainder $R_3(v)$ is of order $\tau^{1+2\gamma}$ for sufficiently smooth functions $v$. More precisely, we define the auxiliary function $g_{\gamma}(x) = \sum_{k \in \Z} \hat{g}_{k,\gamma}  \,\ee^{i k x}$ through its Fourier coefficients
\[
\hat{g}_{k,\gamma} = \left(1+\vert k \vert\right)^{\gamma} \left \vert \hat{v}_k \right \vert.
\]
Together with the estimate
\[
\left \vert k_\ell^2 k_j k_m \right \vert \leq \left (1+ \vert k_\ell\vert\right)^2 \left (1 + \vert k_j \vert\right)\left (1+ \vert k_m\vert\right)
\]
this allows us to express the bound on the remainder given in \eqref{r2ga} as follows
\begin{equation*}
\left \Vert R_3(v) \right \Vert_r \leq c \tau^{1+2\gamma}\big \Vert \left(g_\gamma\right)^2 \, g_{2\gamma}\big\Vert_r,
\end{equation*}
where $0 \leq \gamma \leq 1.$ Together with the bilinear estimate \eqref{bil} this yields the error bound
\begin{equation*}
\begin{aligned}
\left \Vert R_3(v) \right \Vert_r & \leq c \tau^{1+2\gamma} \Vert g_\gamma \Vert_r^2 \Vert g_{2\gamma} \Vert_r \\
&= c \tau^{1+2\gamma} \left( \sum_{k \in \Z} \left(1+\vert k \vert\right)^{2(r+\gamma)} \left \vert \hat v_k\right \vert^2\right)
\left( \sum_{k \in \Z} \left(1+\vert k \vert\right)^{2(r+2\gamma)} \left \vert \hat v_k\right \vert^2\right)^{1/2}
\end{aligned}
\end{equation*}
which implies that
\begin{equation}\label{Rfinal2ga}
\left \Vert R_3(v) \right \Vert_r \leq c  \tau^{1+2\gamma}  \Vert v \Vert_{r+\gamma}^2 \Vert v \Vert_{r+2\gamma}, \qquad 0 \leq \gamma \leq 1.
\end{equation}
This is the sought after bound.
\end{proof}

\subsection{Computation of the integrals}

We are left with solving the integral in \eqref{J1}. The last term of the integral is simply seen\black
$$
-\tau \sum_{k_1,k_2,k_3 \in \Z} \hat{\overline{v}}_{k_1}\hat v_{k_2} \hat v_{k_3}\ee^{i(k_1+k_2+k_3) x} =  -\tau \vert v\vert^2 v.
$$
Due to the \emph{structure of resonances} we will treat the other two integrals
\begin{subequations}\label{J1J2}
\begin{align}
J_{1,x}^\tau(v) &= \sum_{k\in\Z}\sum_{\small \substack{k_1,k_2,k_3 \in \Z\\k_1+k_2+k_3 = k}} \hat{\overline{v}}_{k_1}\hat v_{k_2} \hat v_{k_3}\ee^{i kx} \int_0^\tau \ee^{ 2 i s k_1 k } \dd s \label{eq:J1J2a} \\
J_{2,x}^\tau(v) &=  \sum_{k\in\Z}\sum_{\small \substack{k_1,k_2,k_3 \in \Z\\k_1+k_2+k_3 = k}} \hat{\overline{v}}_{k_1}\hat v_{k_2} \hat v_{k_3}\ee^{i kx} \int_0^\tau  \ee^{2 i s k_2 k_3}\dd s \label{eq:J1J2b}
\end{align}
\end{subequations}
separately. To simplify the presentation we will suppress that $k_1,k_2,k_3 \in \Z$ in all sums.

For the first integral \eqref{eq:J1J2a}, we get the following result.

\begin{lemma}\label{lem:int1}
For $v\in H^r$ we have
\begin{equation*}
\begin{aligned}
J_{1,x}^\tau(v) & =  \frac{i}{2} \left[ \ee^{-i \tau \partial_x^2} \partial_x^{-1}
\left(\bigl(\ee^{-i \tau \partial_x^2} \partial_x^{-1} \overline{v}\bigr) \bigl( \ee^{i \tau \partial_x^2} v^2 \bigr)\right) - \partial_x^{-1} \left( \big(\partial_x^{-1}\overline{v}\big) v^2\right)
\right]\\
&\qquad + \tau \hat{\overline{v}}_0 v^2 + \tau \widehat{\left( \vert v \vert^2 v\right)}_0- \tau  \overline{\hat{v}}_0 \widehat{\left( v^2\right)}_0.
\end{aligned}
\end{equation*}
\end{lemma}

\begin{proof} Respecting the resonance set
\[
\{ k_1,k_2,k_3 \in \Z \;;\;  k_1 = 0 \ \text{or} \ k = k_1+k_2+k_3 = 0\}
\]
and using the representation (see \eqref{key})
\begin{equation}\label{eq:observation}
2k_1 k = k^2 + k_1^2- \left(k_2+k_3\right)^2 \quad \text{for} \ k = k_1+k_2+k_3
\end{equation}
we obtain
\begin{equation}\label{J11}
J_{1,x}^\tau(v) =  \sum_{k\in\Z}\sum_{\small\substack{k_1+k_2+k_3 = k \\ k_1\neq 0,\ k \neq 0,}} \hat{\overline{v}}_{k_1}\hat v_{k_2} \hat v_{k_3} \ee^{i kx}\,
 \int_0^\tau \ee^{ 2 i s k_1 k } \dd s + \tau \hat{\overline{v}}_0 v^2 + \tau \sum_{\small\substack{k_1+k_2+k_3 = 0 \\ k_1\neq 0}}\hat{\overline{v}}_{k_1} \hat{v}_{k_2} \hat{v}_{k_3}.
\end{equation}
Computing the above integral and employing \eqref{eq:observation} shows that
$$
\int_0^\tau \ee^{ 2 i s k_1 k } \dd s = \frac{i}{2}\frac{\ee^{ i \tau \left[k^2 + k_1^2- \left(k_2+k_3\right)^2  \right] }-1}{\left(i k_1\right) \left(i k\right)}.
$$
We further employ
$$
\sum_{\small\substack{k_1+k_2+k_3 = 0 \\ k_1\neq 0}}\hat{\overline{v}}_{k_1} \hat{v}_{k_2} \hat{v}_{k_3} =
\sum_{\small\substack{k_1+k_2+k_3 = 0}}\hat{\overline{v}}_{k_1} \hat{v}_{k_2} \hat{v}_{k_3} - \sum_{\small\substack{k_2+k_3 = 0}} \hat{\overline{v}}_0 \hat{v}_{k_2} \hat{v}_{k_3}.
$$
Taking all this together gives at once the desired result.
\end{proof}

For the second integral \eqref{eq:J1J2b}, we get the following result.

\begin{lemma}\label{lem:int2}
For $v\in H^r$ we have
\begin{equation*}
J_{2,x}^\tau(v)  =  \frac{i}{2} \left[ \ee^{-i \tau \partial_x^2} \left(\partial_x^{-1} \ee^{i \tau \partial_x^2} v\right)^2  - \left(\partial_x^{-1} v\right)^2 \right] \overline v + \tau \hat{v}_0\bigl( 2  v - \hat{v}_0\bigr)\overline v .
\end{equation*}
\end{lemma}

\begin{proof}
Similarly as above, by respecting the resonance set
\[
\{ k_1,k_2,k_3 \in \Z \;;\;  k_2 = 0 \ \text{or} \ k_3 = 0\}
\]
and using the representation (see \eqref{key})
\[
2k_2 k_3 = \left(k_2+k_3\right)^2 - \left(k_2^2 + k_3^2\right)
\]
we obtain the following expression for $J_2^\tau(v)$
\begin{equation}\label{J12}
\begin{aligned}
J_{2,x}^\tau(v) & = \overline{v} \sum_{\small\substack{k_2,k_3 \in \Z}}\hat v_{k_2} \hat v_{k_3} \ee^{i (k_2+k_3)x} \int_0^\tau  \ee^{ 2 i s k_2 k_3}\dd s\\
& = \frac{i}{2} \overline{v} \hspace{-2mm}\sum_{\small\substack{k_2,k_3 \in \Z\\k_2\neq 0,\ k_3 \neq 0}}\hspace{-2mm} \hat v_{k_2} \hat v_{k_3} \ee^{i (k_2+k_3)x} \,\frac{ \ee^{ i \tau \left[\left(k_2+k_3\right)^2 - \left(k_2^2 + k_3^2\right)\right]}-1}{ \left(i  k_2\right)\left(i k_3\right)} + 2 \tau  \overline{v}\hat{v}_{0}\sum_{\small\substack{k\in\Z\\k \neq 0}} \hat{v}_{k} \ee^{i k x} +\tau\overline{v} \hat{v}_{0}^2.
\end{aligned}
\end{equation}
This gives at once the sought after result.
\end{proof}

\subsection{A second-order Fourier integrator for the cubic NLS}\label{sec:FI2nd}

The representation of $v(t_n+\tau)$ given in \eqref{vpre} together with that of $J^\tau_x(\V)$ in \eqref{J1} yields that
\begin{equation}
\begin{aligned}\label{exVf}
v(t_n+\tau) &= v(t_n) - i \mu \ee^{-i t_n \partial_x^2} \Big[J_{1,x}^\tau\big(\ee^{i t_n \partial_x^2} v(t_n)\big)+ J_{2,x}^\tau\big(\ee^{i t_n \partial_x^2} v(t_n)\big)\Big]\\
&\qquad + i\mu \tau \ee^{- i t_n \partial_x^2} \left[ \left\vert \ee^{i t_n \partial_x^2}v(t_n)\right\vert^2 \ee^{i t_n \partial_x^2}v(t_n)
+ i\mu \frac{\tau}{2}\left \vert \ee^{i t_n \partial_x^2} v(t_n) \right \vert^4  \ee^{i t_n \partial_x^2} v(t_n) \right]\\
&\qquad + \mathcal R_{2\gamma}(\tau^{1+2\gamma}).
\end{aligned}
\end{equation}
For the remainder term, we have used that $2\gamma+1\le 2+\gamma$ for $0\le\gamma\le 1$. Using the first terms in the Taylor series expansion $\mathrm{e}^{\tau \lambda} = 1+ \lambda + \frac12 \lambda^2 + \mathcal{O}(\lambda^3)$  we obtain with $\lambda=  i \mu \tau \left\vert \ee^{i t_n \partial_x^2}v(t_n)\right\vert^2$ that
\begin{align*}
v(t_n) + i \mu   \tau  \ee^{- i t_n \partial_x^2} & \left[\left\vert \ee^{i t_n \partial_x^2}v(t_n)\right\vert^2   \ee^{i t_n \partial_x^2}v(t_n)
 + i \mu \frac{\tau}{2} \left \vert \ee^{i t_n \partial_x^2} v(t_n) \right \vert^4  \ee^{i t_n \partial_x^2} v(t_n) \right] \\&
 =
 \ee^{- i t_n \partial_x^2}   \left[ 1+ i \mu \tau \left\vert \ee^{i t_n \partial_x^2}v(t_n)\right\vert^2
 + \frac{1}{2}\left(i \mu \tau  \left\vert \ee^{i t_n \partial_x^2}v(t_n)\right\vert^2\right)^2
 \right]\ee^{i t_n \partial_x^2}v(t_n)
  \\&
 =
 \ee^{- i t_n \partial_x^2}  \mathrm{e}^{ i \mu \tau \left\vert \ee^{i t_n \partial_x^2}v(t_n)\right\vert^2 }\ee^{i t_n \partial_x^2}v(t_n) +  \mathcal R_{0}(\tau^{3}).
\end{align*}
Plugging the above expansion into \eqref{exVf} yields that \black
\black
\begin{equation}\label{eq:expansionV}
\begin{aligned}
v(t_n+\tau) &= \ee^{- i t_n \partial_x^2} \biggl(\ee^{i \mu \tau \left\vert \ee^{i t_n \partial_x^2}v(t_n)\right\vert^2} \ee^{i t_n \partial_x^2}v(t_n)  - i \mu  \Big(J_{1,x}^\tau\big(\ee^{i t_n \partial_x^2} v(t_n)\big)+ J_{2,x}^\tau\big(\ee^{i t_n \partial_x^2} v(t_n)\big)\Big)\biggr)\\
&\qquad + \mathcal R_{2\gamma}(\tau^{1+2\gamma}).
\end{aligned}
\end{equation}
This expansion motivates us to define the following numerical scheme in $v$. Let $\tau$ be the step size and $t_n = n\tau$. The numerical approximation $v^n$ to the solution $v(t_n)$ is given by
\begin{equation}\label{schemeV}
\begin{aligned}
v^{n+1} &=  \ee^{- i t_n \partial_x^2} \biggl(\ee^{i \mu \tau \left\vert \ee^{i t_n \partial_x^2}v^n\right\vert^2} \ee^{i t_n \partial_x^2}v^n- i \mu  \left(J_{1,x}^\tau(\ee^{ i t_n \partial_x^2} v^n) + J_{2,x}^\tau (\ee^{ i t_n \partial_x^2} v^n) \right)\biggr).
\end{aligned}
\end{equation}

To obtain a numerical approximation to the solution $u(t)$ of the original problem~\eqref{nls} we twist the solution back again (see \eqref{twist}), i.e., we set
$$
u^{n+1} = \ee^{ i t_{n+1} \partial_x^2} v^{n+1}
$$
and $v^n = \ee^{ - i t_n \partial_x^2} u^n$, respectively.  We thus obtain the following scheme for the integration of the one-dimensional cubic nonlinear Schr\"odinger equation
\begin{equation}\label{uu}
\begin{aligned}
u^{n+1}
& = \ee^{ i \tau \partial_x^2}  \Bigl(\ee^{i \mu \tau \vert u^n\vert^2}u^n - i \mu \bigl(J_{1,x}^\tau\left( u^n\right) + J_{2,x}^\tau \left( u^n\right)\bigr)\Bigr)
\end{aligned}
\end{equation}
with $J^\tau_{1,x}(u^n)$ and $J^\tau_{2,x}(u^n)$ given by Lemmas~\ref{lem:int1} and \ref{lem:int2}, respectively. This scheme will be called \emph{second-order Fourier integrator} henceforth.

\section{Convergence analysis in one dimension $(d=1)$}
\label{sect:conv-1d}

In this section, we give the convergence results for the one-dimensional scheme~\eqref{uu}. The following lemma gives the required Lipschitz bounds in $H^r$ for the terms arising in \eqref{uu}.

\begin{lemma}\label{lem:lipest}
For given $R>0$ there exists a constant $L>0$ such that
\begin{align}
\left\|\ee^{i\mu\tau |w_1|^2}w_1 - \ee^{i\mu\tau |w_2|^2}w_2\right\|_r &\le (1+\tau L)\|w_1-w_2\|_r, \label{lip-exp}\\
\left\|J_{1,x}^\tau\left( w_1\right) - J_{1,x}^\tau\left( w_2\right)\right\|_r &\le \tau L\|w_1-w_2\|_r, \label{lip-j1}\\[1.5mm]
\left\|J_{2,x}^\tau\left( w_1\right) - J_{2,x}^\tau\left( w_2\right)\right\|_r &\le \tau L\|w_1-w_2\|_r. \label{lip-j2}
\end{align}
for all $w_1,w_2\in H^r$ with $\|w_1\|,\|w_2\|\le R$.
\end{lemma}

\begin{proof}
The proof of \eqref{lip-exp} makes use of the decomposition
$$
\ee^{i\mu\tau |w_1|^2}w_1 - \ee^{i\mu\tau |w_2|^2}w_2 = \ee^{i\mu\tau |w_1|^2}(w_1-w_2) +\left(\ee^{i\mu\tau |w_1|^2}- \ee^{i\mu\tau |w_2|^2}\right)w_2.
$$
The arising products are all estimated with the help of the bilinear estimate~\eqref{bil}. The identity $\|1\|_r=1$ finally implies the desired bound.

From the definition of $J_{1,x}^\tau(v)$ and $J_{2,x}^\tau(v)$ given in \eqref{J1J2} we readily read off their Fourier coefficients. Let $\alpha_k$ and $\beta_k$ denote the Fourier coefficients of $w_1$ and $w_2$, respectively. For showing~\eqref{lip-j1}, we have to estimate the Fourier series with coefficients
\begin{equation*}
\widehat W_k = \sum_{\small \substack{k_1,k_2,k_3 \in \Z\\k_1+k_2+k_3 = k}}\left(\overline\alpha_{k_1}\alpha_{k_2} \alpha_{k_3}- \overline\beta_{k_1}\beta_{k_2} \beta_{k_3}\right) \int_0^\tau \ee^{ 2 i s k_1 k } \dd s.
\end{equation*}
Using the decomposition
$$
\overline\alpha_{k_1}\alpha_{k_2} \alpha_{k_3}- \overline\beta_{k_1}\beta_{k_2} \beta_{k_3} = \overline\alpha_{k_1}\alpha_{k_2} (\alpha_{k_3}-\beta_{k_3}) + \overline\alpha_{k_1}(\alpha_{k_2} - \beta_{k_2}) \beta_{k_3} + (\overline\alpha_{k_1} - \overline\beta_{k_1})\beta_{k_2} \beta_{k_3},
$$
these coefficients are readily estimated as
$$
\bigl|\widehat W_k\bigr| = \tau \left(\widehat V_{1,k} + \widehat V_{2,k} + \widehat V_{3,k}\right)
$$
with
$$
\widehat V_{1,k} = \sum_{k_1+k_2+k_3 = k}|\alpha_{k_1}||\alpha_{k_2}| |\alpha_{k_3}-\beta_{k_3}|
$$
and similar expressions for $\widehat V_{2,k}$ and $\widehat V_{3,k}$. Thus, for $1\le j\le 3$, the Fourier series with coefficients $(\widehat V_{j,k})_{k\in\Z}$ are products of known functions, and the application of the bilinear estimate~\eqref{bil} shows the desired bound.

For \eqref{lip-j2}, we proceed in exactly the same way.
\end{proof}

We are now in the position to state the convergence result in $H^r(\mathbb{T})$ with $r>1/2$.

\begin{theorem}\label{thm-1d}
Let $d=1$, $r>1/2$ and $0 < \gamma \leq 1$. Assume that the exact solution of \eqref{nls} satisfies $u(t) \in H^{r+2\gamma}(\mathbb{T})$ for $0 \leq t \leq T$. Then, there exists a constant $\tau_0>0$ such that for all step sizes $0<\tau \leq \tau_0$ and times $t_n \leq T$ we have that the global error of \eqref{uu} is bounded by
\begin{equation*}
\Vert u(t_{n}) - u^{n} \Vert_r \leq c \tau^{2\gamma},
\end{equation*}
where $c$ depends on $\sup_{0\leq t \leq T} \Vert u(t)\Vert_{r+2\gamma}$.
\end{theorem}
\begin{proof}
First note that as $\mathrm{e}^{i t \partial_x^2}$ is a linear isometry in $H^r$  (see \eqref{iso}) the error in $u$ and $v$ is the same, i.e.,
\[
\Vert u(t_{n}) - u^{n} \Vert_r = \Vert v(t_{n}) - v^{n} \Vert_r.
\]
Thanks to the representations \eqref{eq:expansionV}, \eqref{schemeV} and the form of the remainder $\mathcal R_{2\gamma}(\tau^{1+2\gamma})$ given in \eqref{eq:Rest} we obtain with the help of Lemma~\ref{lem:lipest} that
\begin{equation*}
\Vert v(t_{n+1}) - v^{n+1} \Vert_r \leq \bigl(1+ \tau(1+2 \vert \mu\vert) L \bigr) \Vert v(t_n) - v^n\Vert_{r}   + c \tau^{2\gamma+1},
\end{equation*}
where $L$ is the Lipschitz constant depending on $\Vert v(t_n)\Vert_r$ and $\Vert v^n\Vert_r$. The local error constant $c$ depends on a higher Sobolev norm of the solution $\sup_{0\leq t \leq T} \Vert u(t)\Vert_{r+2\gamma}$.

The assertion then follows by induction, respectively, a \emph{Lady Windermere's fan} argument (see, for example \cite{Faou12,HNW93,Lubich08}).
\end{proof}

Note that in order to exploit the bilinear estimates \eqref{bil} we had to assume that $r>1/2$ in the above theorem. Nevertheless one can derive a similar error bound in $L^2$ following the approach in \cite{Lubich08} as explained in the following. Fix $\varepsilon>0$. In a first step observe that Theorem \ref{thm-1d} implies that the second-order scheme \eqref{uu} is convergent  with order~$\tau^{1/2- \varepsilon}$ in $H^{3/2+\varepsilon}$ for solutions in $H^2$. This yields an a priori bound on the numerical solution $u^n$ in $H^{3/2+\varepsilon}$. Now we use the refined bilinear estimate
\begin{equation}\label{eq:est-refined}
\Vert f g \Vert_{0} \leq c \Vert f \Vert_{0} \Vert g \Vert_{3/2+\varepsilon},
\end{equation}
which is a consequence of H\"{o}lder's inequality $\Vert f g \Vert_{0} \leq \Vert f \Vert_{0} \Vert g \Vert_{L^\infty}$ and the standard Sobolev embedding theorem in dimensions $1\le d\le 3$. Estimate \eqref{eq:est-refined} together with the a priori boundedness of the numerical solution in $H^{3/2+\varepsilon}$ then yields the following error bound in $L^2(\mathbb{T})$.
\begin{corollary}\label{cor:1d}
Let $d=1$ and assume that the exact solution of \eqref{nls} satisfies $u(t) \in H^{2}(\mathbb{T})$ for $0 \leq t \leq T$. Then, there exists a  constant $\tau_0>0$ such that for all  step sizes   $0<\tau \leq \tau_0$ and times $t_n \leq T$ we have that the global error of \eqref{uu} is bounded by
\begin{equation*}
\Vert u(t_{n}) - u^{n} \Vert_0 \leq c \tau^{2},
\end{equation*}
where $c$ depends on $\sup_{0\leq t \leq T} \Vert u(t)\Vert_{2}$.\qed
\end{corollary}

\section{Construction of a scheme in arbitrary dimensions $d\geq 1$}\label{sect:const-dd}

In this section we extend our approach to arbitrary dimensions. Due to the problem of resonances, we have to modify slightly our approach. First, we adapt our notation. For
$\boka=(\kappa_1,\ldots,\kappa_d), \bola=(\lambda_1,\ldots,\lambda_d)\in\Z^d$ and $\bx=(x_1,\ldots,x_d)\in \mathbb{T}^d$, we set
$$
\boka\cdot\bola = \kappa_1 \lambda_1 + \ldots + \kappa_d \lambda_d,\qquad \boka\cdot\bx = \kappa_1 x_1 + \ldots + \kappa_d x_d.
$$
Again, we start from the approximation
\begin{equation}
\begin{aligned}\label{vpre2d}
v(t_n+\tau) &= v(t_n) - i \mu  \ee^{\m i t_n \Delta}  \int_0^\tau \ee^{\m i s\Delta} \Big[\left(\ee^{\m i s\Delta} \overline \V \right)  \left( \ee^{\p i s \Delta} \V \right)^2\Big]\dd s \\
&\qquad - \mu^2  \ee^{- i t_n \Delta}\int_0^\tau s \,\left \vert \V \right \vert^4 \V\,\dd s + \mathcal R_{2\gamma}(\tau^{2+\gamma})\\
&= v(t_n) - i \mu \ee^{- i t_n \Delta}  J^\tau_\bx(\V)
- \mu^2 \frac{\tau^2}{2}\ee^{- i t_n \Delta} \left(\left \vert \V \right \vert^4 \V
\right) + \mathcal R_{2\gamma}(\tau^{2+\gamma}),
\end{aligned}
\end{equation}
which is obtained in exactly the same way as \eqref{vpre} in Section~\ref{sec:sec-order}. The integral $J^\tau_\bx$ has the following form
\begin{equation} \label{Jd2}
J_\bx^\tau(v) = \int_0^\tau \sum_{\boka,\bola,\bonu\in\mathbb Z^d} \ee^{i s \Omega\left(\boka,\bola,\bonu\right)} \hat{\overline {v}}_{\boka} \hat v_{\bola} \hat v_{\bonu} \mathrm{e}^{i (\boka+\bola+\bonu)\cdot\bx}\,\dd s
\end{equation}
with the nonlinear interactions
\[
\Omega\left(\boka,\bola,\bonu\right) = (\boka+\bola+\bonu)\cdot(\boka+\bola+\bonu) + \boka\cdot\boka-\bola\cdot\bola-\bonu\cdot\bonu.
\]
We have to approximate this integral. For this purpose, we simplify
\begin{equation} \label{key2d}
\Omega\left(\boka,\bola,\bonu\right) = 2\,\boka\cdot\boka  + 2\, \boka\cdot\bola + 2\,\boka\cdot\bonu + 2\, \bola\cdot\bonu
\end{equation}
and expand the term $\ee^{i s \Omega\left(\boka,\bola,\bonu\right)}$ by applying Lemma \ref{mani} three times. This shows that
\begin{equation}\label{omfac2}
\begin{aligned}
\ee^{i s \Omega\left(\boka,\bola,\bonu\right)} &= \mathrm{e}^{2is \,\boka\cdot\boka} + \mathrm{e}^{2is\,\boka\cdot\bola} + \mathrm{e}^{2is \boka\cdot\bonu} + \mathrm{e}^{2is \,\bola\cdot\bonu} - 3 + \widetilde R_2^s\left(\boka,\bola,\bonu\right),
\end{aligned}
\end{equation}
where $\widetilde R_2^s$ is such that the remainder
\begin{align}
\widetilde R_3^\tau(v) = \int_0^\tau \sum_{\boka,\bola,\bonu\in\mathbb Z^d} \widetilde R_2^s\left(\boka,\bola,\bonu\right) \hat{\overline v}_{\boka} \hat v_{\bola}  \hat v_{\bonu} \mathrm{e}^{i (\boka+\bola+\bonu)\cdot\bx}\,\dd s
\end{align}
satisfies the following bounds:
\begin{equation}\label{BRk3-2da}
\widetilde R_3^\tau(v) = \mathcal R_{1+2\gamma}(\tau^{2+\gamma}),\qquad 0\le \gamma\le 1
\end{equation}
(use $\delta=\gamma$ and $\theta=1$ in Lemma~\ref{mani}) and
\begin{equation}\label{BRk3-2db}
\widetilde R_3^\tau(v) = \mathcal R_{\gamma}(\tau^{1+\gamma}),\qquad 0\le \gamma\le 1
\end{equation}
(use $\delta=0$ and $\theta=\gamma$ in Lemma~\ref{mani}). These bounds are verified in exactly the same way as the corresponding bound in the one-dimensional case in Lemma \ref{lem:int-rem}.

It remains to compute the integrals that arise by inserting \eqref{omfac2} into \eqref{Jd2}. We start with the quadratic term
\begin{equation*}
L(w) = \sum_{\boka\in\Z^d} \hat w_{\boka} \ee^{i\,\boka\cdot \bx} \int_0^\tau \ee^{2is\,\boka\cdot\boka} \dd s
\end{equation*}
for which we have the following result.

\begin{lemma}\label{lem:integral1}
For $w\in L^2(\mathbb T^d)$, it holds
\begin{equation}\label{eq:intq}
L(w) = \tau\varphi_1(-2i\tau\Delta) w,
\end{equation}
where $\varphi_1(z) = \left(\ee^z-1\right)/z$.
\end{lemma}

\begin{proof}
Noting that
$$
\int_0^\tau \ee^{2is\,\boka\cdot\boka}\,\dd s = \tau \varphi_1(2i\tau\,\boka\cdot\boka),
$$
the result follows at once.
\end{proof}

For the computation of the remaining integrals, we employ Lemma \ref{mani} several times in \eqref{omfac2}. This results in
$$
\mathrm{e}^{2is\,\boka\cdot\bola} + \mathrm{e}^{2is\,\boka\cdot\bonu} + \mathrm{e}^{2is \,\bola\cdot\bonu} =
\sum_{j=1}^d \left(\ee^{2is\kappa_j\lambda_j} + \ee^{2is\kappa_j\nu_j} + \ee^{2is\lambda_j\nu_j} \right)- 3(d-1) + R_2^s\left(\boka,\bola,\bonu\right),
$$
where $R_2^s$ is such that the remainder
\begin{align}\label{BRk3-2d2}
R_3^\tau(v) = \int_0^\tau \sum_{\boka,\bola,\bonu\in\mathbb Z^d} R_2^s\left(\boka,\bola,\bonu\right) \hat{\overline v}_{\boka} \hat v_{\bola}  \hat v_{\bonu} \mathrm{e}^{i (\boka+\bola+\bonu)\cdot\bx}\,\dd s
\end{align}
is of class $\mathcal R_{2\gamma}(\tau^{1+2\gamma})$.

Motivated by this decomposition, we consider for $1\le j\le d$ the term
\begin{equation}
K_j(w,v) = \sum_{\boka,\bola\in\Z^d} \hat w_{\boka} \hat v_{\bola} \ee^{i(\boka+\bola)\cdot\bx} \int_0^\tau \ee^{2is\kappa_j\lambda_j}\,\dd s.
\end{equation}
This term is symmetric in its arguments $(w,v)$ and has the following representation.
\begin{lemma}\label{lem:integral-kompli}
For $w,v\in H^r(\mathbb{T}^d)$, we have
\begin{equation*}
\begin{aligned}
K_j(w,v) &= \frac{i}2\left[\ee^{-i\tau \partial_j^2}\left( \bigl(\ee^{i\tau \partial_j^2} \partial_j^{-1} w\bigr) \bigl( \ee^{i\tau \partial_j^2} \partial_j^{-1} v\bigr) \right) - \bigl(\partial_j^{-1} w\bigr) \bigl(\partial_j^{-1} v\bigr) \right]\\
&\quad +\tau \bigl[v \,\hat w_{0,j} + w \,\hat v_{0,j} - \hat w_{0,j} \hat v_{0,j}\bigr],
\end{aligned}
\end{equation*}
where $\hat w_{0,j}$ denotes the 0th Fourier coefficient of the partial Fourier transform of $w$ in direction~$j$.
\end{lemma}

\begin{proof}
The proof is almost identical to that of Lemma~\ref{lem:int2} and therefore omitted.
\end{proof}

Lemma~\ref{lem:integral-kompli} implies at once the following result: for $1\le j\le d$ and all
$v\in H^r$ it holds
\begin{equation}\label{eq:int-all}
\begin{aligned}
K_j(\overline v, v) &= \frac{i}2\left[\ee^{-i\tau \partial_j^2}\left( \bigl( \ee^{i\tau \partial_j^2} \partial_j^{-1}\overline v\bigr) \bigl( \ee^{i\tau \partial_j^2} \partial_j^{-1}v \bigr) \right) - \bigl|\partial_j^{-1} v\bigr|^2 \right]+\tau \bigl[ \hat v_{0,j}\overline v + \hat{\overline v}_{0,j} v - \left|\hat v_{0,j}\right|^2\bigr],\\
K_j(v, v) &= \frac{i}2\left[\ee^{-i\tau \partial_j^2}\left( \bigl( \ee^{i\tau \partial_j^2} \partial_j^{-1}v\bigr)^2 \right) - \bigl(\partial_j^{-1} v\bigr)^2 \right]+\tau \bigl[ 2\hat v_{0,j} v - \left(\hat v_{0,j}\right)^2\bigr].
\end{aligned}
\end{equation}
This motivates us to define the numerical approximation $v^{n+1}$ to the solution $v(t_{n+1})$ in $d\ge 2$ dimensions as follows
\begin{equation}\label{schemeV2d}
\begin{aligned}
v^{n+1} &= \ee^{- i t_n \Delta} \Bigl[\ee^{i \mu \tau \left\vert V_n\right\vert^2} V_n + i\mu \tau (3d-1)\bigl( \left\vert V_n\right\vert^2 V_n\bigr)\Bigr.\\
&\qquad - i \mu \Bigl. \tau\varphi_1(-2i\tau\Delta)\left(|V_n|^2 V_n\right) - i \mu \textstyle \sum_{j=1}^d \Bigl(K_j(V_n,V_n)\overline V_n + 2 K_j(\overline V_n,V_n)V_n\Bigr)\Bigr],
\end{aligned}
\end{equation}
where we have set $V_n = \ee^{ i t_n \Delta} v^n$. For the original solution $u(t_{n+1})$ we in obtain by twisting back the variable, i.e., setting
$$
u^{n+1} = \ee^{ i t_{n+1} \Delta} v^{n+1}
$$
the following numerical scheme
\begin{equation}\label{schemed}
\begin{aligned}
u^{n+1} &= \ee^{i \tau \Delta} \left[\ee^{i \mu \tau \left\vert u^n\right\vert^2} u^n + i\mu \tau (3d-1)\bigl( \left\vert u^n\right\vert^2 u^n\bigr)\right.\\
&\qquad - i \mu \Bigl. \tau\varphi_1(-2i\tau\Delta)\left(|u^n|^2 u^n\right) - i \mu \textstyle\sum_{j=1}^d \Bigl(K_j(u^n,u^n)\overline u^n + 2 K_j(\overline u^n,u^n)u^n\Bigr)\Bigr]
\end{aligned}
\end{equation}
with $K_j$ given in \eqref{eq:int-all}. In the next section we state the error estimates of this scheme.

\section{Convergence analysis of the modified approach}
\label{sect:conv-dd}

The convergence result of the modified approach in arbitrary dimensions is very similar to that of method~\eqref{uu} in one dimension. The only difference comes from the local error terms \eqref{BRk3-2da}, \eqref{BRk3-2db}, which are less favourable than~\eqref{Rfinal2ga} due to resonances.  Again, we first state the convergence result in $H^r(\mathbb{T}^d)$ with $r>d/2$.

\begin{theorem}\label{thmd-a}
Let $r>d/2$ and $0 \le \gamma \leq 1$. Assume that the exact solution of \eqref{nls} satisfies $u(t) \in H^{r+1+2\gamma}(\mathbb{T}^d)$ for $0 \leq t \leq T$. Then, there exists a  constant $\tau_0>0$ such that for all  step sizes $0<\tau \leq \tau_0$ and times $t_n \leq T$ we have that the global error of \eqref{schemed} is bounded by
\begin{equation*}
\Vert u(t_{n}) - u^{n} \Vert_r \leq c \tau^{1+\gamma},
\end{equation*}
where $c$ depends on $\sup_{0\leq t \leq T} \Vert u(t)\Vert_{r+1+2\gamma}$.
\end{theorem}

This theorem is proved exactly in the same way as Theorem~\ref{thm-1d} above. Note that the stronger regularity assumption is due to condition~\eqref{BRk3-2da}, which itself is a consequence of resonances. If the solution is in $H^{r+2}$, our scheme~\eqref{schemed} is convergent of order 3/2. For second-order convergence in $H^r$, our approach requires $H^{r+3}$ regularity.

The following theorem gives a convergence result for solutions in $H^{r+\gamma}$ for $0<\gamma\le 1$. Instead of condition~\eqref{BRk3-2da}, it makes use of condition~\eqref{BRk3-2db}. Otherwise, it is proved in exactly the same way as Theorem~\ref{thmd-a}.

\begin{theorem}\label{thmd-b}
Let $r>d/2$ and $0 \le \gamma \leq 1$. Assume that the exact solution of \eqref{nls} satisfies $u(t) \in H^{r+\gamma}(\mathbb{T}^d)$ for $0 \leq t \leq T$. Then, there exists a  constant $\tau_0>0$ such that for all  step sizes $0<\tau \leq \tau_0$ and times $t_n \leq T$ we have that the global error of \eqref{schemed} is bounded by
\begin{equation*}
\Vert u(t_{n}) - u^{n} \Vert_r \leq c \tau^{\gamma},
\end{equation*}
where $c$ depends on $\sup_{0\leq t \leq T} \Vert u(t)\Vert_{r+\gamma}$.
\end{theorem}

Note that for solutions in $H^{r+\gamma}$, $0<\gamma\le 1$ the order of convergence in one dimension is the same as in higher dimensions. We further recall that the refined bilinear estimate $\Vert f g \Vert_{0} \leq c \Vert f \Vert_{0} \Vert g \Vert_{3/2+\varepsilon}$ holds in spatial dimensions $1 \leq d \leq 3$. Therefore, we can extend our convergence result again to $L^2(\mathbb{T}^d)$.

\begin{corollary}\label{cord-2d}
Let $1 \leq d \leq 3$. Assume that the exact solution of \eqref{nls} satisfies $u(t) \in H^{2}(\mathbb{T}^d)$ for $0 \leq t \leq T$. Then, there exists a  constant   $\tau_0>0$ such that for all  step sizes $0<\tau \leq \tau_0$ and times $t_n \leq T$ we have that the global error of \eqref{schemed} is bounded by
\begin{equation*}
\Vert u(t_{n}) - u^{n} \Vert_0 \leq c \tau^{3/2},
\end{equation*}
where $c$ depends on $\sup_{0\leq t \leq T} \Vert u(t)\Vert_{2}$.
\end{corollary}

\section{Numerical examples}\label{sect:num}

The aim of this section is to illustrate the numerical behavior of our Fourier integrator \eqref{uu}. As a test problem, we choose the cubic nonlinear Schr\"{o}dinger equation~\eqref{nls} with $\mu=1$ on the one-dimension torus $\mathbb{T}$. As our approach relies on discrete Fourier techniques, the interval $[0,2\pi]$ representing the torus is equidistantly discretized with the grid points $x_j=jh$, $0\le j\le N-1$ for $N=2^\ell$ and $h = 2\pi/N$. We integrate this equation in time on the interval $[0,1]$ with constant step size $\tau$ and measure the global error at $T=1$ in a discrete $H^r$ norm for $r=0$ or $r=1$. The results obtained with our Fourier integrator are compared with those obtained by standard Strang splitting.

For a function $u\in L^2(\mathbb T)$, we denote $U_j = u(x_j)$ for $j=0,\ldots N-1$ and define its discretization on the grid by $U=(U_0,\ldots,U_{N-1})^{\sf T}$. For this vector, the discrete $L^2$ norm is defined by
$$
\|U\|_0^2 = h \sum_{j=0}^{N-1} |U_j|^2.
$$
The $H^1$ norm of $u$ also includes its first derivative
\begin{equation}\label{eq:first-der}
\partial_x u(x) = \sum_{k\in\Z} ik\hat u_k \ee^{ikx}.
\end{equation}
A discretization of this derivative can be computed by approximating~\eqref{eq:first-der} with the help of the discrete Fourier transform, i.e.~by multiplying the discrete Fourier coefficients $\hat U_k$ of $U$ for $k=-\frac{N}2,\ldots,\frac{N}2-1$ by $ik$ and then transforming back. Let $V$ be the result of this computation. Then, we define the discrete $H^1$ norm of $U$ by
$$
\|U\|_1^2 = h \sum_{j=0}^{N-1} \left( |U_j|^2 + |V_j|^2\right).
$$
Initial data in $H^r$ are computed by choosing uniformly distributed random numbers in the interval $[-1,1]$ for the real and imaginary part of the $N$ Fourier coefficients, respectively. These coefficients are then divided by $(1+|k|)^{r+1/2}$ for $k=-\frac{N}2,\ldots,\frac{N}2-1$ and finally transformed back with the discrete Fourier transform to get the desired discrete initial data in physical space.

In this section, we carry out two different type of experiments. First, we study the full error at time $T = n\tau =1$ for various values of $N$. The results for initial data in $H^5$ are given in Figure~\ref{fig:full1-2}, those for initial data in $H^3$ are given in Figure~\ref{fig:full3-4}. For smooth initial data, both integrators are second-order convergent in time and their performance compares quite well. For initial data in $H^3$, however, Strang spitting shows a very irregular behavior and order reduction occurs whereas the Fourier integrators is still second-order convergent. Note that the spatial discretization error behaves as $N^{-4}$ in Figure~\ref{fig:full1-2} and as $N^{-2}$ in Figure~\ref{fig:full3-4}. Such a behavior is well expected (see \cite[Prop.~IV.14]{Faou12})\black, but not further analysed in this paper.

\begin{figure}[h!]
\centering
\includegraphics[width=0.48\linewidth]{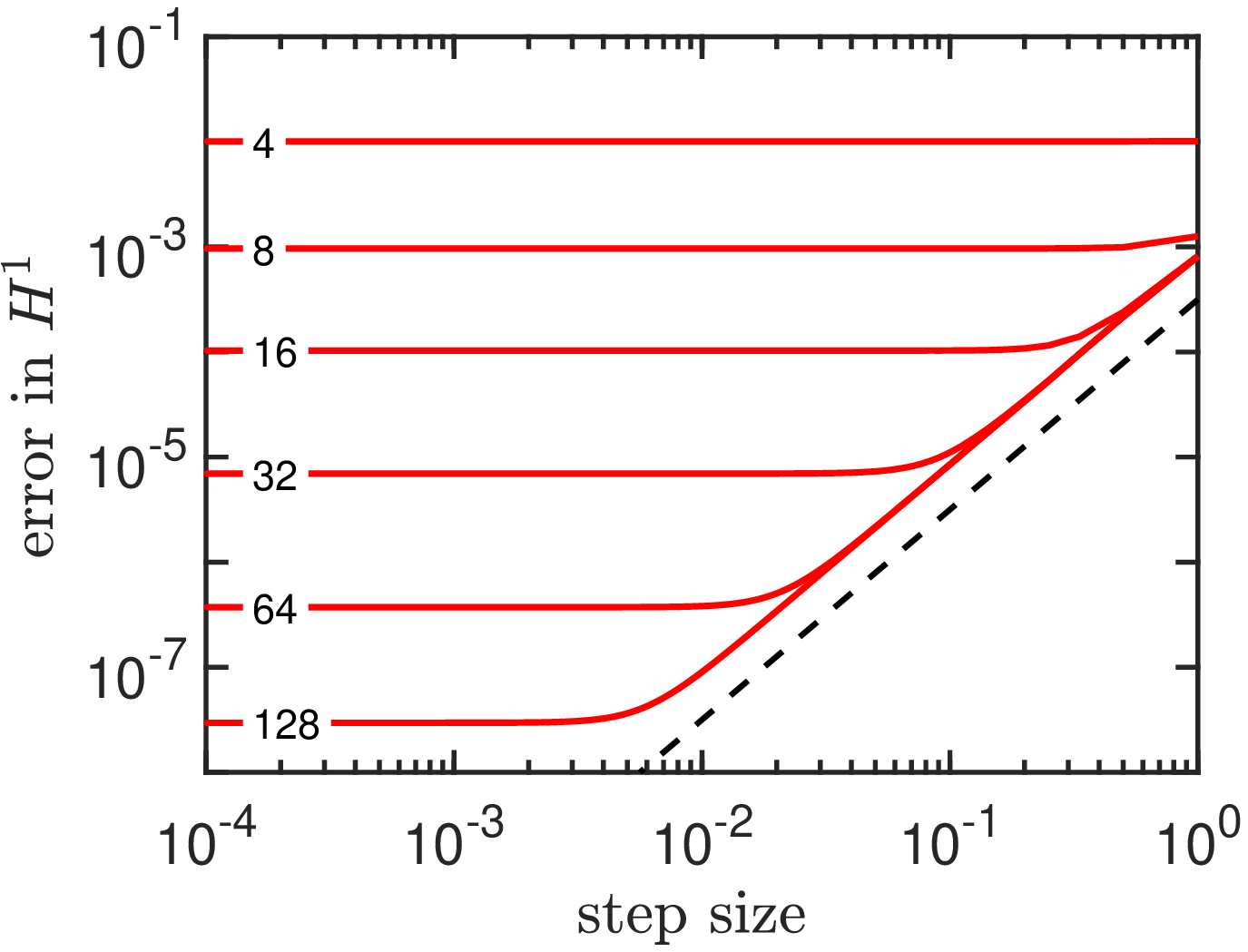}
\hfill
\includegraphics[width=0.48\linewidth]{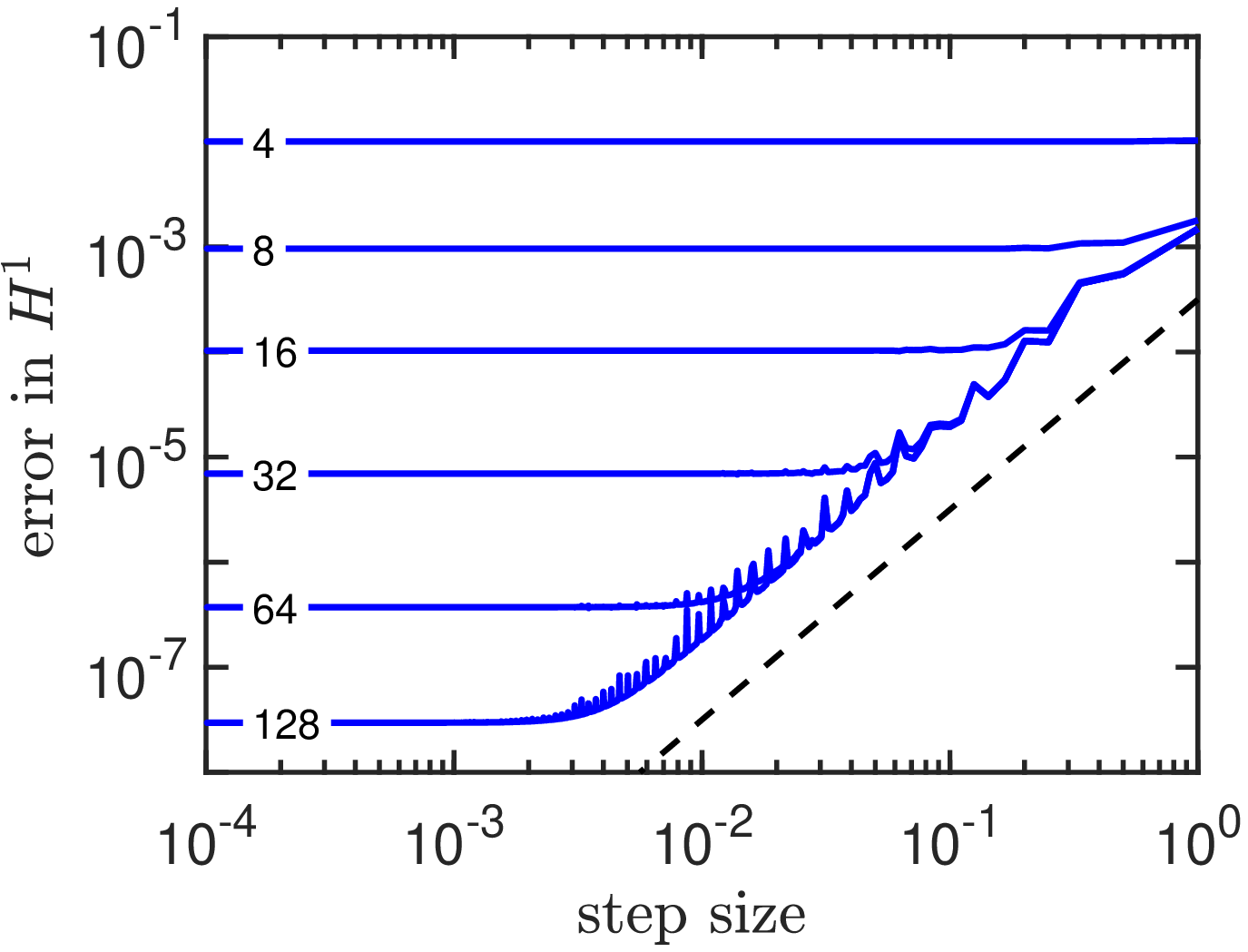}
\caption{Full error of the Fourier integrator (left, red) and classic Strang splitting (right, blue) at $T=1$ for smooth initial data in $H^5$ as a function of the step size $\tau$ and the number $N$ of discretization points in space (small numbers in the graphs). The dashed line has slope two.}\label{fig:full1-2}
\end{figure}

\begin{figure}[h!]
\centering
\includegraphics[width=0.48\linewidth]{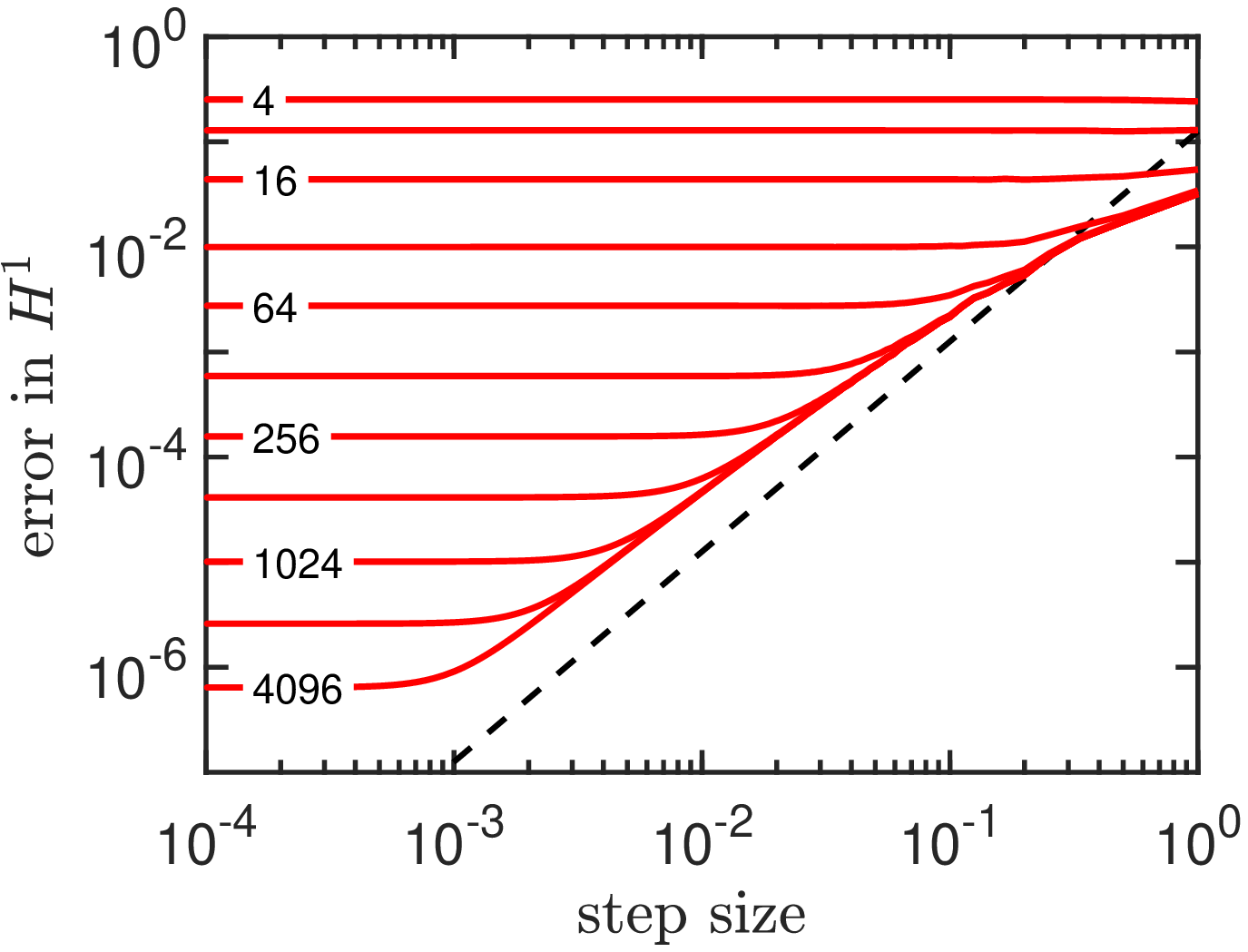}
\hfill
\includegraphics[width=0.48\linewidth]{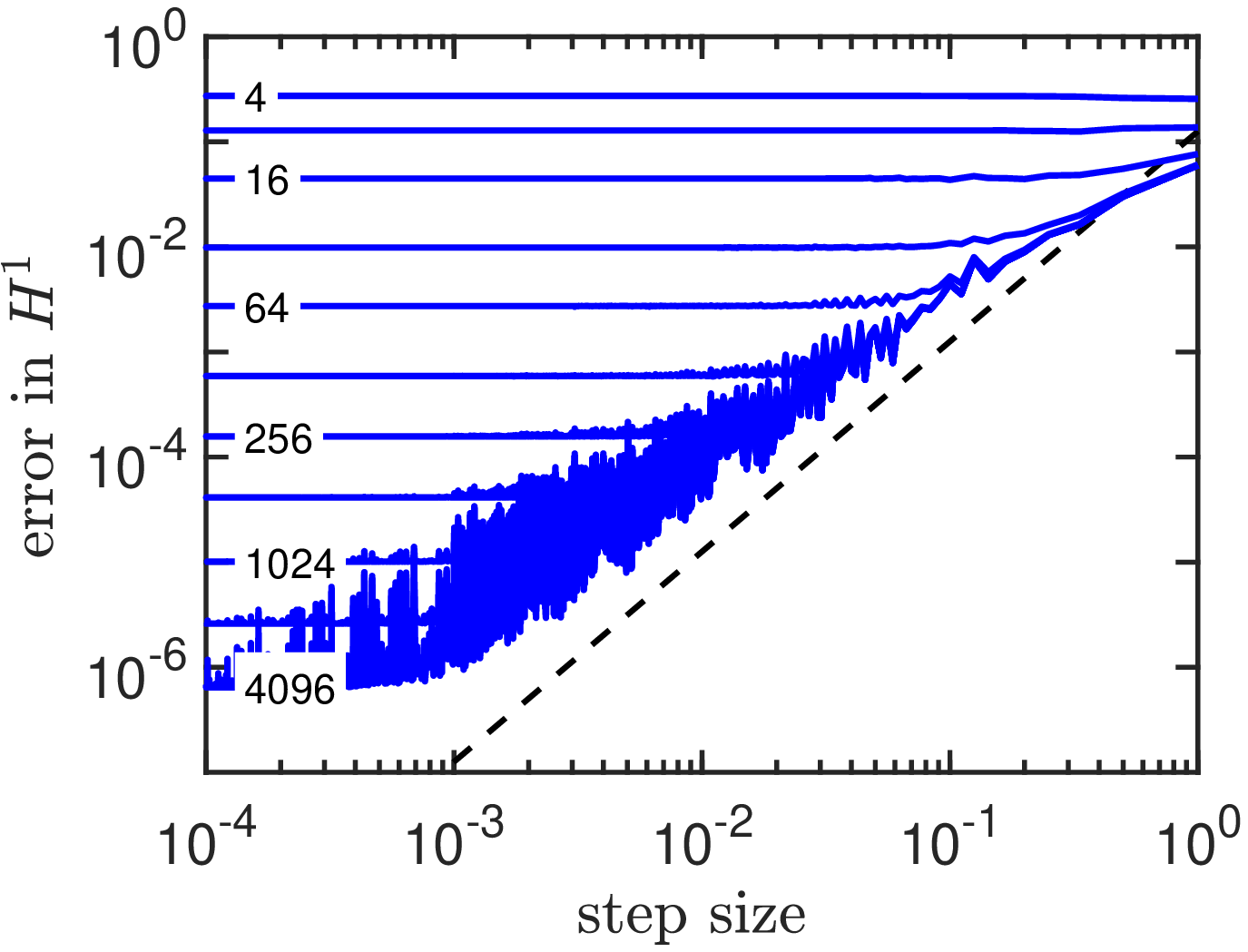}
\caption{Full error of the Fourier integrator (left, red) and classic Strang splitting (right, blue) at $T=1$ for non-smooth initial data in $H^3$ as a function of the step size $\tau$ and the number $N$ of discretization points in space (small numbers in the graphs). The dashed line has slope two.}\label{fig:full3-4}
\end{figure}

\begin{figure}[t]
\centering
\includegraphics[width=0.48\linewidth]{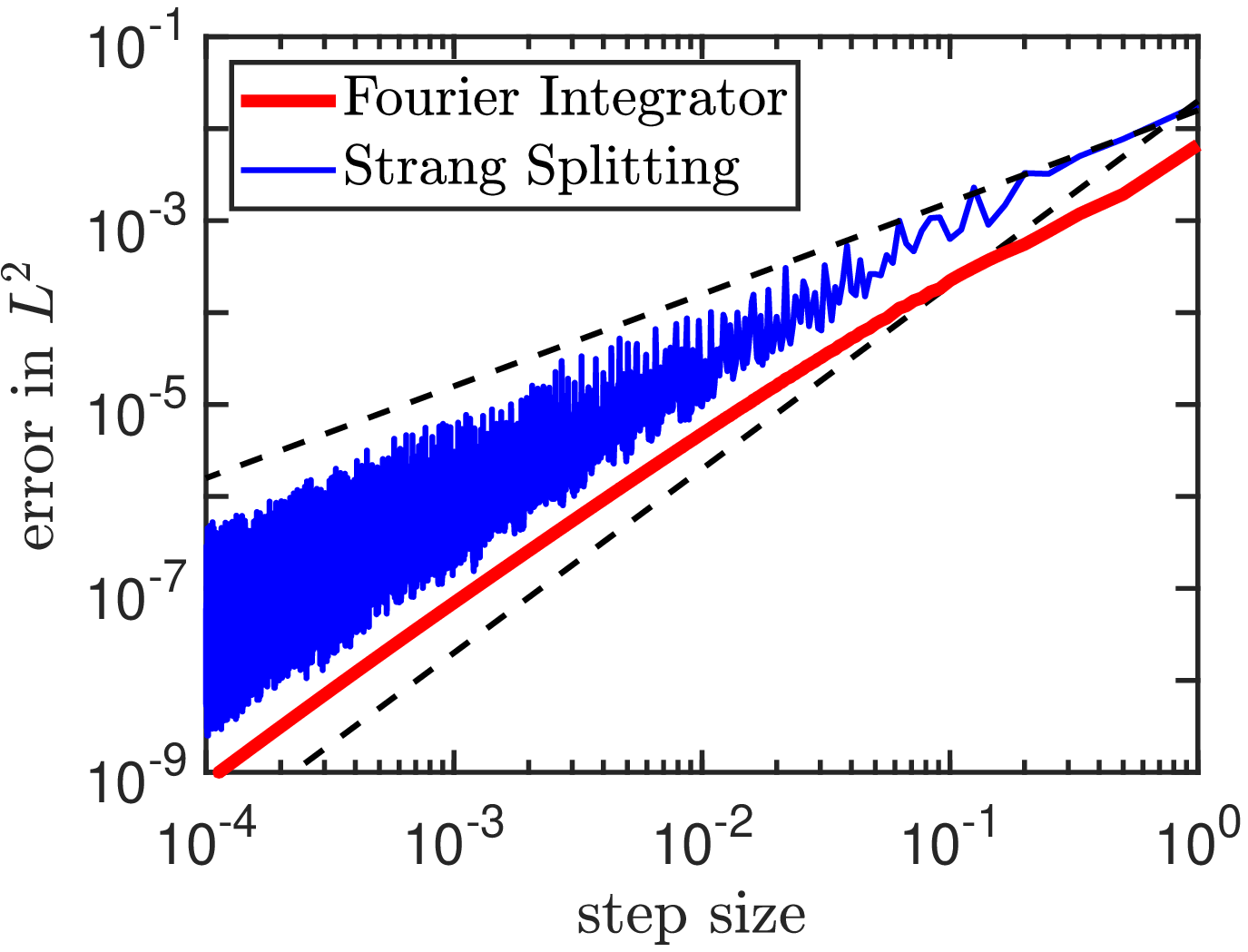}
\hfill
\includegraphics[width=0.48\linewidth]{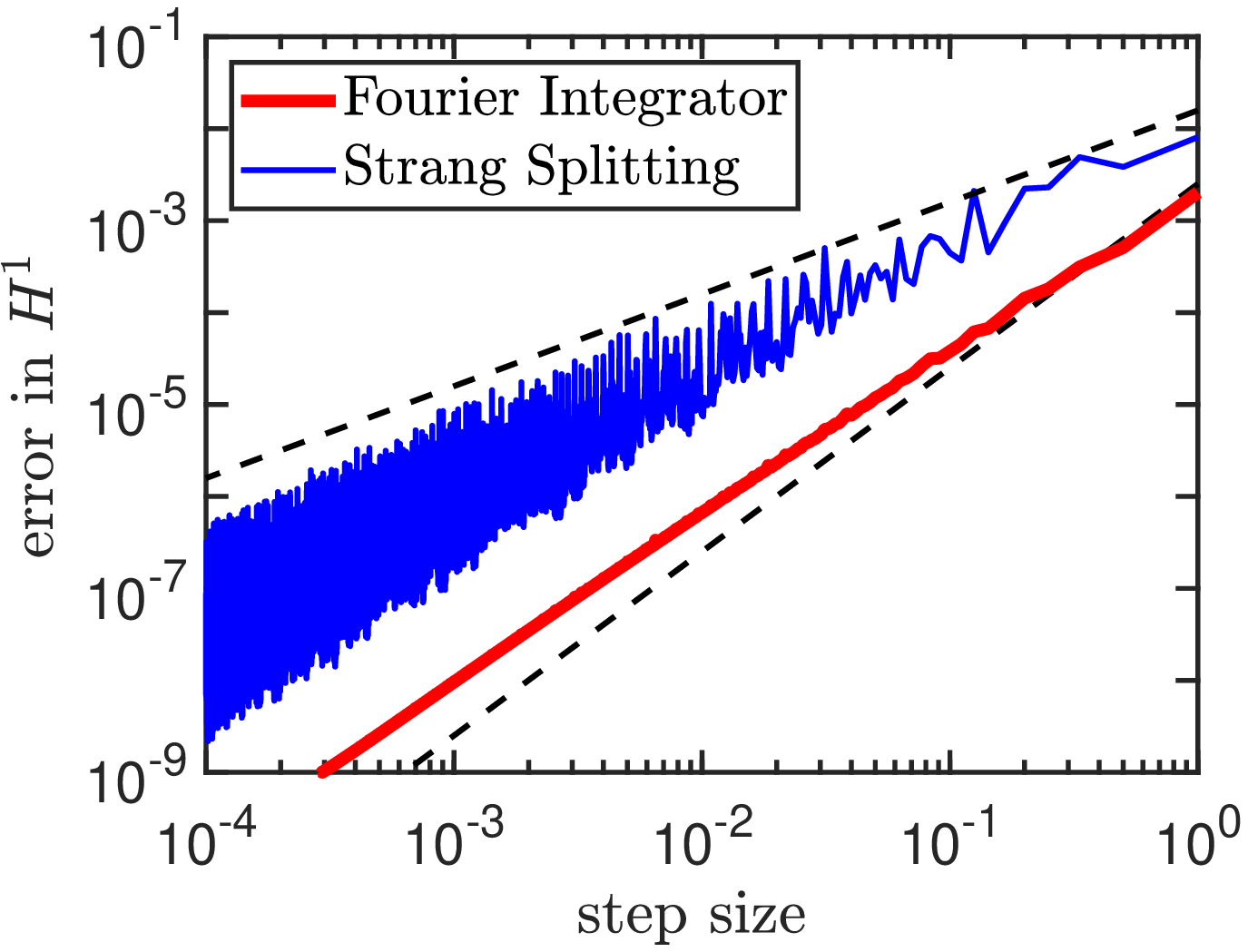}
\caption{Time error of the Fourier integrator (thick line, red) and classic Strang splitting (thin line, blue) at $T=1$. The error is plotted as a function of the step size $\tau$ for a fixed number $N=2^{12}$ of space discretization points. On the left, the error is measured in $L^2$ with initial data in $H^2$; on the right, the error is measured in $H^1$ with initial data in $H^3$. The broken lines are of slope one and two, respectively.}\label{fig:bilder1-2}
\end{figure}

In the second experiment, we study the time error only. The spatial discretization is kept fixed here with $N=2^{12}$. As a reference solution for the Fourier integrator, we take a very accurate solution obtained with Strang splitting, and vice versa. Figure~\ref{fig:bilder1-2} displays on the left the errors in $L^2$ for initial data in $H^2$ and on the right the errors in $H^1$ with initial data in $H^3$. For such non-smooth initial data, the order of Strang splitting is reduced to one whereas the Fourier integrator is still second-order convergent. This behavior of the Fourier integrator was proved in Theorem~\ref{thm-1d} and Corollary~\ref{cor:1d}. For smoother initial data, however, Strang splitting recovers its usual performance. This is shown in Figure~\ref{fig:bilder3-4} for initial data in $H^4$ (left) and $H^5$ (right). The errors in this figure are measured in $H^1$.

\begin{figure}[t]
\centering
\includegraphics[width=0.48\linewidth]{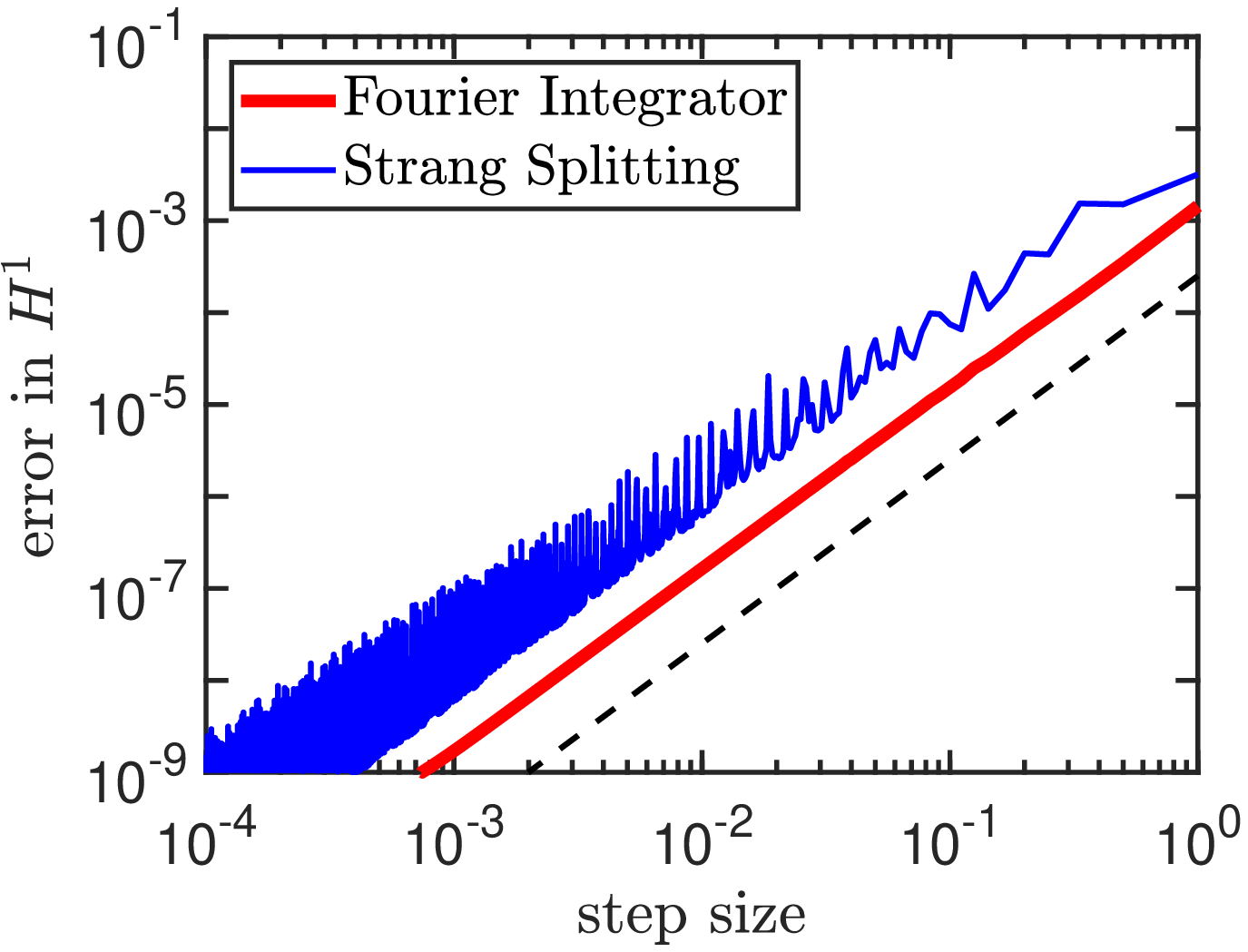}
\hfill
\includegraphics[width=0.48\linewidth]{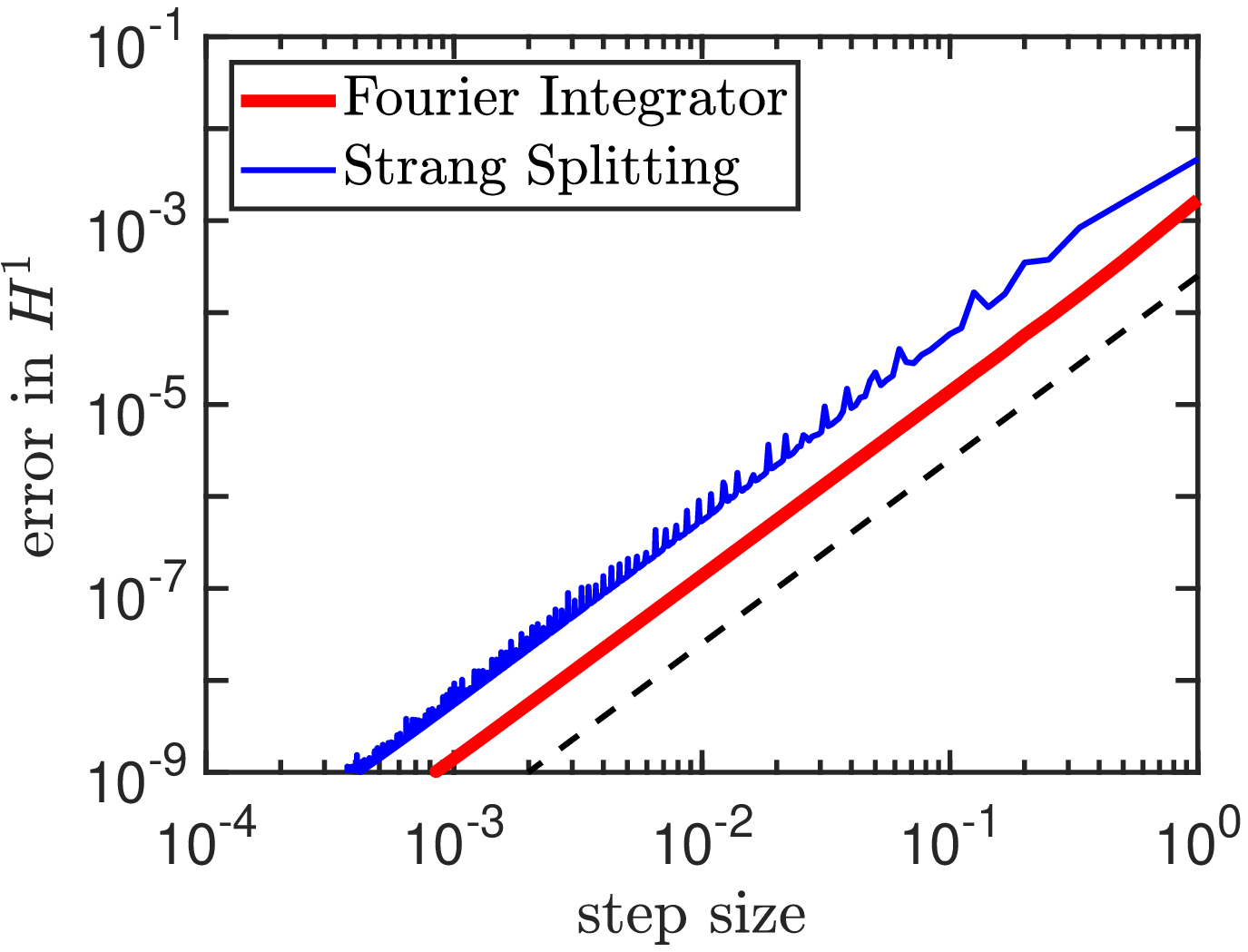}
\caption{Time error of the Fourier integrator (thick line, red) and classic Strang splitting (thin line, blue) at $T=1$. The error is plotted as a function of the step size $\tau$ for a fixed number $N=2^{12}$ of space discretization points. On the left, the error is measured in $H^1$ with initial data in $H^4$; on the right, the error is measured in $H^1$ with initial data in $H^5$. The broken line is of slope two.}\label{fig:bilder3-4}
\end{figure}

\begin{figure}[t]
\centering
\includegraphics[height=4.2cm]{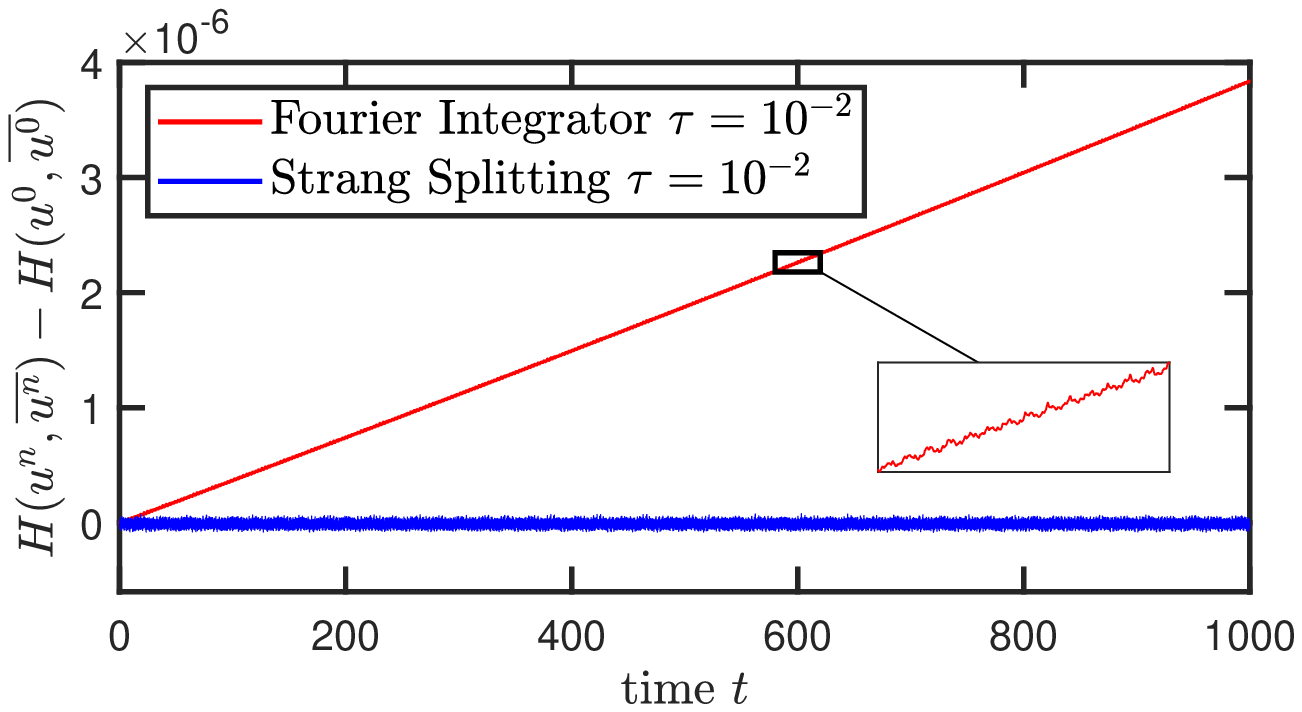}
\hfill
\includegraphics[height=4.2cm]{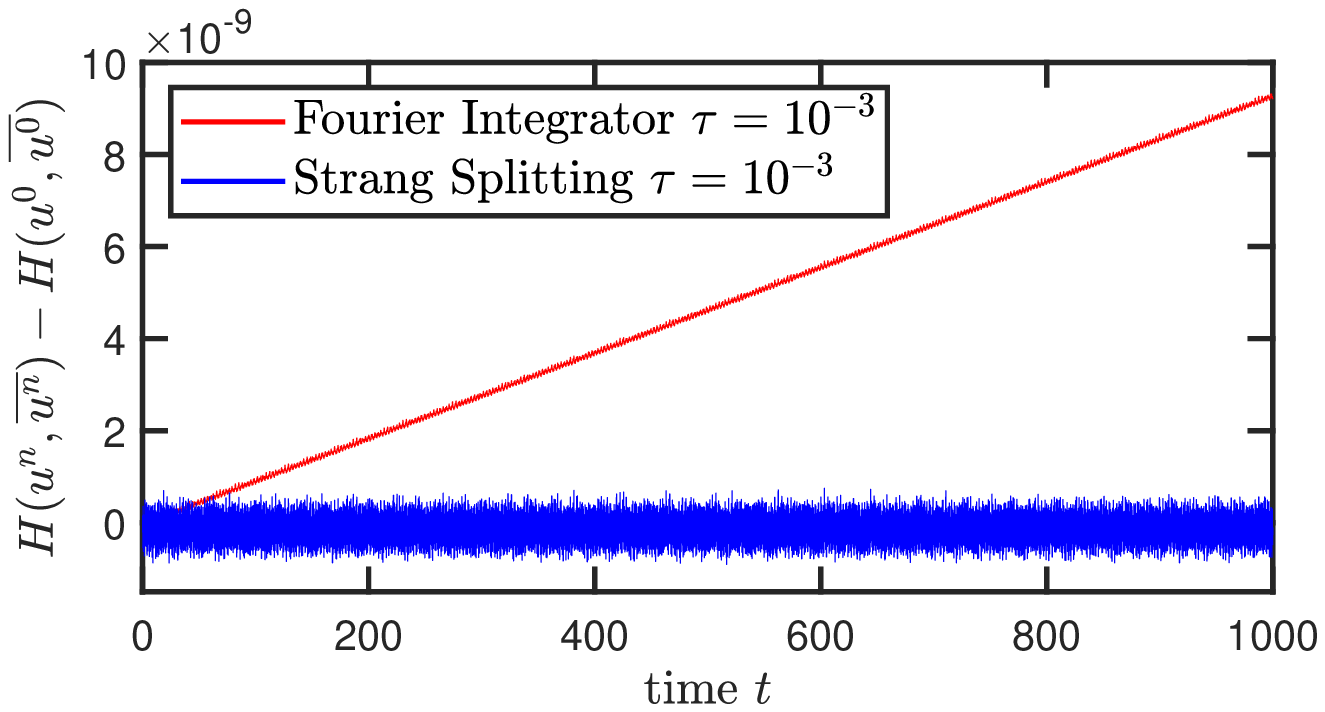}

\vspace{4mm}
\includegraphics[height=6.3cm]{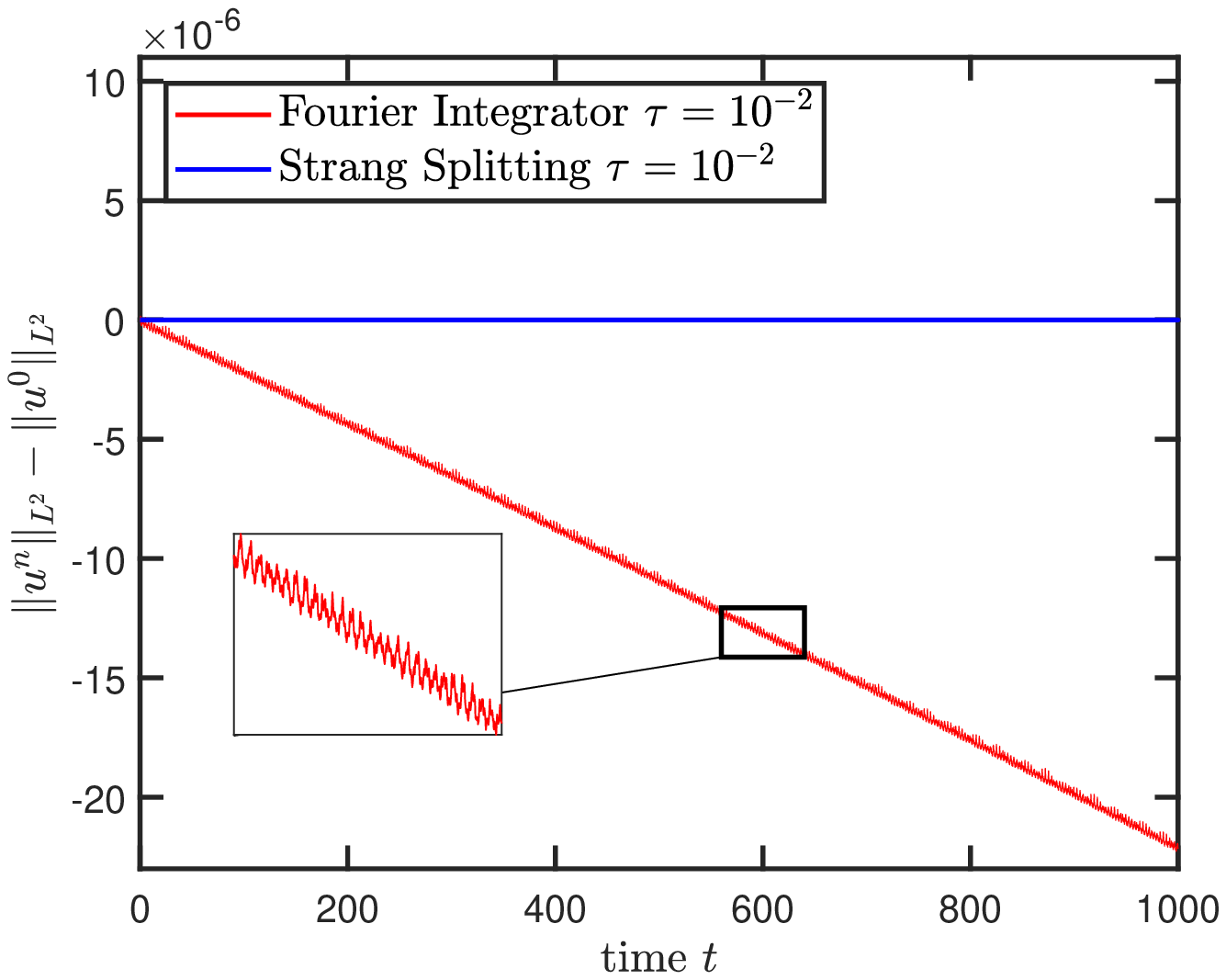}
\hfill
\includegraphics[height=6.3cm]{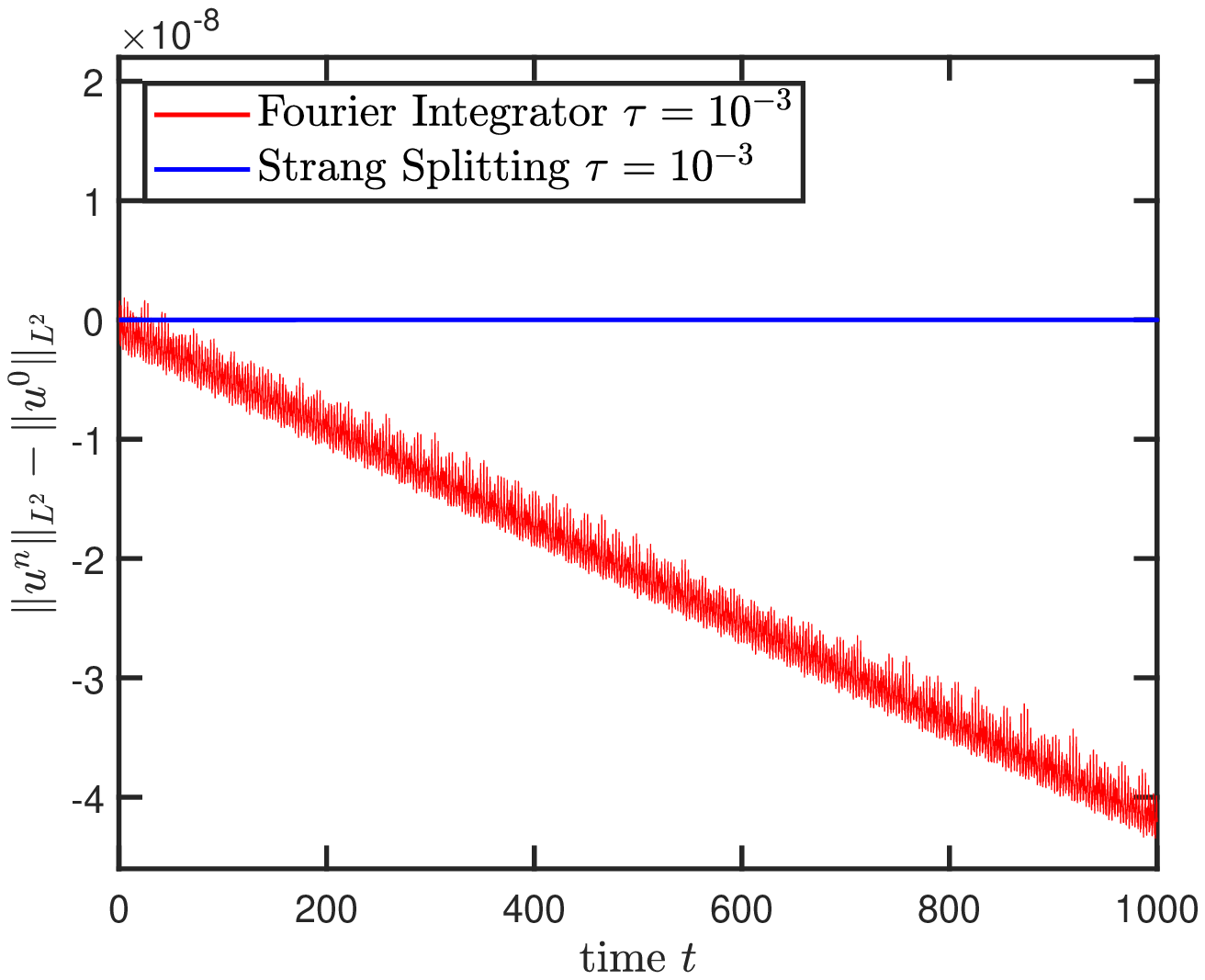}
\caption{Numerical preservation of energy and mass. The upper two figures display the total energy of the numerical solution for an $H^3$ initial value; the lower two figures show the mass of the numerical solution for an $H^2$ initial value. Both experiments were carried out for $N=2^{12}$ grid points and show very good preservation properties.\black }\label{fig:norm_and_energie}
\end{figure}

Note that the Schr\"odinger equation has a Hamiltonian structure preserving the total energy of the system
\begin{equation}\label{ham}
H\big(u(t),\overline{u}(t)\big) = \frac{1}{(2\pi)^d} \int_{\mathbb{T}^d} \left(\vert \nabla u(t,x)\vert^2 + \frac{\mu}{2}\vert u(t,x)\vert^4\right)\mathrm{d}x = H\big(u(0),\overline{u}(0)\big)
\end{equation}
as well as the $L^2$ norm of the exact solution
\begin{equation}\label{l2}
\Vert u(t)\Vert_{L^2}^2 = \Vert u(0)\Vert_{L^2}^2.
\end{equation}
In Figure~\ref{fig:norm_and_energie} we investigate the numerical structure preservation of these quantities under our Fourier integrator \eqref{uu}: we simulate the evolution of the total energy $H\big(u^n,\overline{u^n}\big)$ and mass $\Vert u^n\Vert_{L^2}^2$ up to $t_n = n\tau = 1000$. Note that Strang splitting preserves the mass, but does not preserve the energy exactly. Our Fourier integrator \eqref{uu} also shows very good conservation properties. Although there is a linear growth in mass and energy, the corresponding coefficient is very small and decays fast with decreasing $\tau$. A thorough analysis of this behavior is considered as future work.\black
\section{Conclusions}\label{sect:conclusions}

In this paper, we developed a new Fourier integrator for the numerical solution of the cubic nonlinear Schr\"{o}dinger equation with non-smooth initial data. The new integrator compares well with classic Strang splitting for smooth initial data. Note, however, that Strang splitting requires less fast Fourier transforms per time step and is thus preferable in this situation. On the other hand, for non-smooth initial data, the new integrator is much more reliable and efficient compared to Strang splitting which delivers irregular results and is prone to order reduction.


\begin{thebibliography}{}

\bibitem{Bour93}
{\rm J. Bourgain},
{\em Fourier transform restriction phenomena for certain lattice subsets and applications to nonlinear evolution equations. Part I: Schr\"odinger equations.} Geom. Funct. Anal. 3, 209--262 (1993).

\bibitem{BeDe02}
{\rm C. Besse, B. Bid\'egaray, S. Descombes}, {\em  Order estimates in time of splitting methods for
the nonlinear Schr\"odinger equation}, SIAM J. Numer. Anal. 40, 26--40 (2002).

\bibitem{CanG15}
{\rm B. Cano, A. Gonz\'alez-Pach\'on},
{\em Exponential time integration of solitary waves of cubic Schr\"odinger equation.} Appl. Numer. Math. 91, 26--45 (2015).
%
%
\bibitem{CCO08}
{\rm E. Celledoni, D. Cohen, B. Owren},
{\em Symmetric exponential integrators with an application to the cubic Schr\"odinger equation.} Found. Comput. Math. 8, 303--317 (2008).

\bibitem{CoGa12}
{\rm D. Cohen, L. Gauckler},
{\em One-stage exponential integrators for nonlinear Schr\"odinger equations over long times.} BIT 52, 877--903 (2012).
%
\bibitem{Duj09}
{\rm G. Dujardin},
{\em Exponential Runge-Kutta methods for the Schr\"odinger equation.} Appl. Numer. Math. 59, 1839--1857 (2009).

\bibitem{ESS16}
{\rm J. Eilinghoff, R. Schnaubelt, K. Schratz},
{\em Fractional error estimates of splitting schemes for the nonlinear Schr\"odinger equation.} J. Math. Anal. Appl. 442, 740--760 (2016).

\bibitem{Faou12}
{\rm E. Faou},
{\em Geometric Numerical Integration and Schr\"odinger Equations.}
European Math. Soc. Publishing House, Z\"urich 2012.
%
\bibitem{HNW93}
{\rm E. Hairer, S. P. N\o rsett, G. Wanner},
{\em Solving Ordinary Differential Equations I. Nonstiff Problems.}
Second edition. Springer, Berlin 1993.
%
\bibitem{HLW}
{\rm E. Hairer, C. Lubich, G. Wanner},
{\em Geometric Numerical Integration. Structure-Preserving Algorithms for Ordinary Differential Equations.} Second edition, Springer, Berlin 2006.

\bibitem{HochOst10}
{\rm M. Hochbruck, A. Ostermann},
{\em Exponential integrators.}  Acta Numer. 19, 209--286 (2010).

\bibitem{HoS16}
{\rm M. Hofmanova, K. Schratz},
{\em An exponential-type integrator for the KdV equation.} Numer. Math. 136, 1117--1137 (2017).

\bibitem{HLRS10}
{\rm H. Holden, K. H. Karlsen, K.-A. Lie, N. H. Risebro},
{\em Splitting for Partial Differential Equations with Rough Solutions.} European Math. Soc. Publishing House, Z\"urich 2010.

\bibitem{Ignat11}
{\rm L. I. Ignat},
{\em A splitting method for the nonlinear Schr\"odinger equation.} J. Differential Equations 250, 3022--3046 (2011).

\bibitem{JaLu00}
{\rm T. Jahnke, C. Lubich},
{\em Error bounds for exponential operator splittings}. BIT 40, 735--744  (2000).

\bibitem{KT05}
{\rm A.-K. Kassam, L. N. Trefethen},
{\em Fourth-order time-stepping for stiff PDEs.} SIAM J. Sci. Comput. 26, 1214--1233 (2005).

\bibitem{Law67}
{\rm J. D. Lawson},
{\em Generalized Runge--Kutta processes for stable systems with large Lipschitz constants.} SIAM J. Numer. Anal. 4, 372--380 (1967).

\bibitem{Lubich08}
{\rm C. Lubich},
{\em On splitting methods for {S}chr\"{o}dinger--{P}oisson and cubic nonlinear {S}chr\"{o}dinger
equations.} Math. Comp. 77, 2141--2153 (2008).

\bibitem{McLacQ02}
{\rm R.I. McLachlan, G.R.W. Quispel},
{\em Splitting methods},
Acta Numer. 11, 341--434 (2002).

\bibitem{OSch2017}
{\rm A. Ostermann, K. Schratz},
{\em Low regularity exponential-type integrators for semilinear {S}chr\"{o}dinger equations.}
Found. Comput. Math. 18, 731--755 (2018).

\bibitem{Tao06}
{\rm T. Tao},
{\em Nonlinear Dispersive Equations. Local and Global Analysis.} Amer. Math. Soc., Providence 2006.

\bibitem{Ta12}
{\rm M. Thalhammer},
{\em Convergence analysis of high-order time-splitting pseudo-spectral methods for nonlinear Schr\"odinger equations.} SIAM J. Numer. Anal. 50, 3231--3258 (2012).



\end{thebibliography}
\end{document}